\documentclass[12pt,a4paper,leqno]{article}

\usepackage{amsmath,amssymb,amsthm,amscd,eepic}
\usepackage{latexsym}
\usepackage{geometry}
\usepackage{graphicx}
\usepackage{eepic,color}

\newcommand{\BS}[1]{\boldsymbol{#1}}
\newcommand{\BB}[1]{\mathbb{#1}}
\makeatletter
\renewcommand{\section}{%
  \@startsection{section}%
   {1}%
   {\z@}%
   {-3.5ex \@plus -1ex \@minus -.2ex}%
   {2.3ex \@plus.2ex}%
   {\normalfont\normalsize\bfseries}%
}%
\makeatother

\numberwithin{equation}{section}		
\makeatletter
\newcommand{\subjclass}[2][2010]{%
  \let\@oldtitle\@title%
  \gdef\@title{\@oldtitle\footnotetext{#1 \emph{Mathematics subject classification.} #2}}%
}
\newcommand{\keywords}[1]{%
  \let\@@oldtitle\@title%
  \gdef\@title{\@@oldtitle\footnotetext{\emph{Key words and phrases.} #1.}}%
}
\makeatother
\newtheorem{thm}{Theorem}[section]

\newtheorem{lem}[thm]{Lemma}

\theoremstyle{definition}

\newtheorem{rem}[thm]{Remark}
\newtheorem{exam}[thm]{Example}

\makeatletter
\let\c@equation\c@thm
\let\cl@equation\cl@thm

\makeatother

\renewenvironment{itemize}%
{%
   \begin{list}{\parbox{10pt}{$\bullet$}}
   {%
      \setlength{\topsep}{0pt}
      \setlength{\itemindent}{5pt}
      \setlength{\leftmargin}{10pt}
      \setlength{\rightmargin}{0pt}
      \setlength{\labelsep}{6pt}
      \setlength{\labelwidth}{10pt}
      \setlength{\itemsep}{0em}
      \setlength{\parsep}{0em}
      \setlength{\listparindent}{0pt}
   }
}{%
   \end{list}%
}

\newcommand{\CR}[1]{{\color{red}{#1}}} %

\newcommand{\comp}{\textrm{\scriptsize$\circ$}}

\newcommand{\sign}[1]{\mathop{\mathrm{sign}}{#1}}
\newcommand{\ord}[1]{\mathop{\mathrm{ord}}{#1}}

\title{Curvature criteria of $\mathcal A$-simple singularities $\BB{R},0\longrightarrow\BB{R}^2,0$ and their parallel curves}
\author{Toshizumi Fukui and Saiki Hoshino\\
\footnotesize Saitama University, 255 Shimo-Okubo, Saitama, 338-8570, Japan\\
\footnotesize\textit{E-mail address}:
\texttt{tfukui@rimath.saitama-u.ac.jp}, 
\texttt{s.hoshino.math@gmail.com}
}
\date{December 27, 2025}
\allowdisplaybreaks[4]

\begin{document}
\maketitle

\begin{abstract}
We introduce the notion of curvature parameters for singular plane curves with finite multiplicities and define the notion of curvatures for them. 
We then provide criteria to determine their singularity types for $\mathcal A$-simple singularities. 
As an application, we investigate singularity types of their parallel curves.  
\end{abstract}
Plane curves have long been studied since the early development of geometry. However, their differential-geometric structure at singular points has received relatively little attention. 
This is partly due to the divergence of classical curvature at singular points, and to the intrinsic difficulty of classifying degenerate singularities. 
Meanwhile, in singularity theory, the concept of $\mathcal A$-simple singularities was introduced and  $\mathcal A$-simple curve singularities have attracted considerable attention from many researchers, and numerous studies have been devoted to their classification (\cite{A}, \cite{GH}, \cite{JP}, \cite{N}, \cite{S}, \cite{Z}). 
In this paper, we focus $\mathcal A$-simple curve singularities from the differential geometric perspective.

We introduce a curvature parameter for locally irreducible plane curve germs and define invariants that extend the classical curvature of nonsingular curves. 
In principle, these invariants serve as a tool for identifying the singularity type.
Our primary aim is to formulate criteria (Theorem \ref{Thm:S}) for determining the singularity type of $\mathcal{A}$-simple singularities in terms of this invariant. 
As an application, we determine all cases in which the parallel curve to an $\mathcal{A}$-simple singularity is also 
$\mathcal{A}$-simple.

It is natural to model plane curves as the images of smooth maps $\phi: \mathbb{R} \longrightarrow \mathbb{R}^2$. 
In the nonsingular case, the standard construction
--- 
reparametrising the curve by arc length and defining curvature as the derivative of the angle of the tangent vector with respect to arc length
---
is well established in classical texts. 
However, this approach does not extend naturally to curves with singularities.

Assuming finite multiplicity, we introduce a parameter $s$ such that $s^m / m!$ represents the arc length, where $m$ denotes the multiplicity of the curve. This allows us to define a curvature-like invariant even in the presence of singularities. 
We refer to such a parameter $s$ as a {\bf curvature parameter}. The notion originates in \cite{Fukui2017}, where it is shown that the associated invariant determines the curve germ up to rotation and translation. In particular, a fundamental theorem is established for plane curve singularities of finite multiplicity. 
In the case of multiplicity two, a related construction has been given in \cite{SU}. 
We also note that this idea appears in the work of Porteous on cusps \cite[\S1.6]{Porteous},
who explicitly stated the conditions for $A_2$ and  $E_6$ singularities (ibid.~page 12). 

The classification of singularities requires the choice of an equivalence relation. We adopt $\mathcal{A}$-equivalence, wherein two map germs are considered equivalent if they are related by coordinate changes in both source and target. This is the standard framework in singularity theory and differential topology. It is known that 
$\mathcal{A}$-equivalence classes may possess moduli in the presence of degenerate singularities, that is, the equivalence class may vary continuously with parameters. Consequently, the classification problem is, in general, highly non-trivial.

Accordingly, we restrict attention to $\mathcal{A}$-simple singularities $\BB{R},0\longrightarrow\BB{R}^2,0$ --- those for which any local deformation gives rise to only finitely many $\mathcal{A}$-equivalence classes. 
The classification of such singularities is due to Bruce and Gaffney (\cite{BG}),  
and is reproduced as Theorem \ref{Thm:ClassificationASimple} in the real case.
The main contribution of this paper is a characterization of singularity types for
$\mathcal A$-simple singularities in terms of the curvature invariant introduced above. 
This result is summarized as Theorem \ref{Thm:S}. 
In particular, as stated in Theorem \ref{Thm:S} (i) (see also Theorem \ref{Thm:M2}),  
the condition characterizing an $A_{2k}$ singularity is especially concise and stands in marked contrast to the more elaborate conditions previously established for the $A_4$ and $A_6$ cases (see \cite[Theorem 1.23]{Porteous} and  \cite[Theorem A.1]{HHM}). 
The criteria for $E_{6k}$, $E_{6k+2}$, $W_{12}$, $W^\#_{1,2q-1}$ and $W_{18}$ singularities (appeared in Theorems \ref{Thm:S} (ii) and (iii)) are also simple, though slightly more intricate (see Theorems \ref{Thm:M3} and \ref{Thm:M4} also).

As an application, we determine the conditions under which the singularities of the parallel curves of an $\mathcal{A}$-simple singularity remain $\mathcal{A}$-simple.
In particular, we show that the parallels of the $E_{12}$ and $E_{14}$ singularities degenerate at the distance 
$\delta $ so that $\delta^{-1}$ is equal to the quantity \eqref{SecondCurv} in Remark~\ref{SexCurv}. 
This quantity can be regarded as an analogue of curvature. 
We also discuss the generalization of this phenomenon to $E_{6k}$ and $E_{6k+2}$ singularities (Theorem \ref{E6kandE6k+2}).

The paper is organized as follows. In \S\ref{1}, we recall some preliminaries from singularity theory, including the classification of \( \mathcal{A} \)-simple singularities due to Bruce and Gaffney. In \S\ref{2}, we introduce the notion of curvature parameters for singular plane curves of finite multiplicity and collect several facts required for subsequent sections. The main results are presented in \S\ref{3}, where we provide criteria for each \( \mathcal{A} \)-simple singularity type, formulated in terms of our invariants\CR{.} The corresponding proofs are given in \S\ref{4}. In \S\ref{5}, we investigate the singularity types of parallel curves associated with given \( \mathcal{A} \)-simple singularities.

\begin{flushleft}
{\bf Contents}
\end{flushleft}
\begin{itemize}
\item[1.] Preliminary\dotfill\pageref{1}
\item[2.] Curvature parameter\dotfill\pageref{2}
\item[3.] Criteria of singularities\dotfill\pageref{3}
\item[4.] Proofs of criteria\dotfill\pageref{4}
\item[5.] Parallel curves\dotfill\pageref{5}
\end{itemize}
\section{Preliminary}\label{1}
In this section, we reviews the basics of singularity theory required in this paper.

We say that two map germs $f,g:(\BB{K}^n,0)\longrightarrow(\BB{K}^p,0)$ are
\begin{itemize}
\item {\bf $\mathcal R$-equivalent} if there is a diffeomorphism germ $h:(\BB{K}^n,0)\longrightarrow(\BB{K}^n,0)$ such that $f\comp h(x)=g(x)$;
\item {\bf $\mathcal L$-equivalent} if there is a diffeomorphism germ $\psi:(\BB{K}^p,0)\longrightarrow(\BB{K}^p,0)$ such that $f(x)=\psi\comp g(x)$;
\item {\bf $\mathcal A$-equivalent} if there are 
a diffeomorphism germs $h:(\BB{K}^n,0)\longrightarrow(\BB{K}^n,0)$ and 
$\psi:(\BB{K}^p,0)\longrightarrow(\BB{K}^p,0)$ such that $f\comp h(x)=\psi\comp g(x)$;  
\item {\bf $\mathcal K$-equivalent} if there are 
a diffeomorphism germ $h:(\BB{K}^n,0)\longrightarrow(\BB{K}^n,0)$ and  
a $C^\infty$-germ $A:(\BB{K}^n,0)\longrightarrow\mathrm{GL}(\BB{K}^p)$ 
such that $A(x) f\comp h(x)=g(x)$. 
\end{itemize}

We  recall the classification result of $\mathcal A$-simple germs $\BB{K},0\longrightarrow\BB{K}^2,0$, due to 
Bruce and Gaffney \cite{BG}, where $\BB{K}=\BB{R}$, $\BB{C}$. 
The notion of $\mathcal A$-simple map is defined as follows (see Definition 2.6 (2) ibid.): 
A map $\phi:\BB{K},0\longrightarrow\BB{K}^2,0$ is {\bf $\mathcal A$-simple} if $\phi$ is an irreducible parametrization 
and for any $k$-parameter deformation $\{\phi^u\}_{u\in\BB{K}^k,0}$, $\phi^0=\phi$, 
we do not have $0\in\BB{K}^k$  in the closure of a set of $u$ with $\phi^u$ all $\mathcal A$ distinct.
\begin{thm}[{\cite[Theorem 3.8]{BG}}]\label{Thm:ClassificationASimple}
The following are representatives of the $\mathcal A$-simple germs $\phi:\BB{R},0\longrightarrow\BB{R}^2,0$:
\begin{center}
\begin{tabular}{c|l}
\hline
Type of $f$ &Normal form of $\phi$\\
\hline$A_{2k}$&$(t^2,t^{2k+1})$\\
\hline $E_{6k}$&$(t^3,t^{3k+1}+\varepsilon_p t^{3(k+p)+2})$, $0\le p\le k-2$; $(t^3,t^{3k+1})$\\
\hline $E_{6k+2}$&$(t^3,t^{3k+2}+\varepsilon_{p+1} t^{3(k+p)+4})$, $0\le p\le k-2$; $(t^3,t^{3k+2})$\\
\hline $W_{12}$&$(t^4,t^5\pm t^7)$, $(t^4,t^5)$\\
\hline $W_{1,2q-1}^{\#}$&$(t^4,t^6+t^{2q+5})$, $q\ge1$\\
\hline $W_{18}$&$(t^4,t^7\pm t^9)$, $(t^4,t^7\pm t^{13})$, $(t^4,t^7)$\\
\hline
\end{tabular}
\end{center}
where $f:(\BB{R}^2,0)\longrightarrow(\BB{R},0)$ is a defining equation of the image of $\phi$
and $\varepsilon_p$ is $1$ if $p$ is even; $\pm1$ if $p$ is odd.
Here type of $f$ is $A_{2k}$, $E_{6k}$, $E_{6k+2}$, $W_{12}$, $W_{1,2q-1}^\#$ or $W_{18}$ means $f$ is $\mathcal R$-equivalent to the normal form in the table below. 
\end{thm}
Theorem 3.8 in \cite{BG} is stated for the complex case. However, as noted on page 465, line 5, \lq\lq The real case is similar, but we omit the details."  The statement given above provides the corresponding result in the real setting. 
It should also be noted that in \cite{BG}, the normal forms of the 
$E$-series are written without parentheses in the exponents of the final terms; the omitted parentheses should be understood as if they were included.
\begin{exam}
The map $t\mapsto(t^m,t^p-t^q)$ is $\mathcal A$-equivalent to $t\mapsto(t^m,t^p+t^q)$ if $p\not\equiv q\bmod2$.
This equivalence is achieved by changing the signs of the coordinates of the source and the target.
\end{exam}

The notation $A_{2k}$, $E_{6k}$, $E_{6k+2}$, $W_{12}$, $W^\#_{1,2q-1}$ and $W_{18}$ come from 
the celebrated classification result of singularities of functions $f:\BB{R}^2,0\longrightarrow\BB{R},0$ due to V.\,I.\,Arnold (see \cite{AGV1}) 
 by $\mathcal{R}$-equivalence. We recall their normal forms here. 
\begin{center}
\begin{tabular}{c|l|c|c|c}
\hline Type &Normal form&$\mu$&$m$&page of \cite{AGV1}\\
\hline $A_{2k}$&$y^2+x^{2k+1}$&$2k$&$0$&246\\
\hline $E_{6k}$&$y^3+x^{3k+1}+(a_0+\cdots+a_{k-2}x^{k-2})x^{2k+1}y$&$6k$&$k-1$&248\\
\hline $E_{6k+2}$&$y^3+x^{3k+2}+(a_0+\cdots+a_{k-2}x^{k-2})x^{2k+2}y$&$6k+2$&$k-1$&248\\
\hline $W_{12}$&$y^4+x^5+cx^3y^2$&$12$&$1$& 247 \\
\hline $W^\#_{1,2q-1}$&$(y^2+x^3)^2+(b_0+b_1x)x^{q+4}y$, $b_0\ne0$&$2q+14$&$2$&247\\
\hline $W_{18}$&$y^4+x^7+(b_0+b_1x)x^4y^2$&$18$&$2$& 248 \\
\hline
\end{tabular}
\end{center}
Here $\mu$ denotes the Milnor number and $m$ denotes the modality with respect to $\mathcal R$-equivalence. 

\begin{lem}[{Real version of \cite[Lemma 2.2]{BG}}]\label{BG}
If $\phi_i:\BB{R},0\longrightarrow\BB{R}^2,0$ are irreducible parameterizations with defining equations 
$f_i:\BB{R}^2,0\longrightarrow\BB{R},0$, $i=0,1$,
then $\phi_0$ and $\phi_1$ are $\mathcal A$-equivalent if and only if $f_0$ and $f_1$ are $\mathcal K$-equivalent.
\end{lem}
\begin{proof}
Let $I_f$ denote the ideal generated by the function $f:\BB{R}^2,0\longrightarrow\BB{R},0$.
It is proved in \cite[p.\,149]{G}  that $f_0$ and $f_1$ are $\mathcal K$-equivalent if and only if
there is a germ of a diffeomorphism $h:\BB{R}^2,0\longrightarrow\BB{R}^2,0$ taking $I_{f_0}$ to $I_{f_1}$, 
that is, such that $h^*I_{f_0} = I_{f_1}$. 
It follows that if $I_{f_0}$ and $I_{f_1}$ are prime ideals then $f_0$ and $f_1$ are $\mathcal K$ equivalent 
if and only if there is a germ of a diffeomorphism $h:\BB{R}^2,0\longrightarrow\BB{R}^2,0$ taking $f_0^{-1}(0),0$ to $f_1^{-1}(0),0$. 

If $\phi_0$ and $\phi_1$ are $\mathcal A$-equivalent then there is a diffeomorphism 
$h:\BB{R}^2,0\longrightarrow\BB{R}^2,0$ preserving the images of $\phi_0$ and $\phi_1$, 
and hence the zero locus of $f_0$ and $f_1$. 
Since $I_{f_j}=\langle f_j\rangle$ is the ideal of germs vanishing in the irreducible germ
$\{f_j=0\},0$ it is prime and so the $f_j$ are $\mathcal K$-equivalent. 

Conversely if $f_0$ and $f_1$  are $\mathcal K$-equivalent there is a germ of 
a diffeomorphism $h:\BB{R}^2,0\longrightarrow\BB{R}^2,0$ taking $\{f_0=0\}, 0$ to $\{f_1=0\},0$. 
Thus the composite $h\comp\phi_0$ is an irreducible parametrization of $\{f_1=0\},0$, 
and since such a parametrization is unique up to change of coordinates [9, p. 96] 
we find that $\phi_0$ and $\phi_1$ are $\mathcal A$-equivalent.
\end{proof}

\section{Curvature parameter}\label{2}
We introduce the notion of curvature parameter for an irreducible curve germ in $\BB{R}^2,0$.

We say that a $C^\infty$-map $\phi:\BB{R},0\longrightarrow\BB{R}^2,0$ is {\bf of multiplicity $m$} at $t=0$, 
if there exists a $C^\infty$-map $\hat{\phi}:\BB{R},0\longrightarrow\BB{R}^2,0$ so that 
$$
\phi(t)=\frac{t^m}{m!}\hat{\phi}(t),\quad \hat{\phi}(0)\ne0.
$$
\begin{lem}
If a $C^\infty$-map $\phi:\BB{R},0\to\BB{R}^2,0$ is of multiplicity $m$, then there is $C^\infty$-parameter $s=s(t)$ so that $\pm s^m/m!$ is an arc length parameter. Moreover there exists a unit vector field $\BS{t}$ along the curve $\phi$ so that 
\begin{equation}\label{PT}
\frac{d\phi}{ds}=\frac{s^{m-1}}{(m-1)!}\,\BS{t}.
\end{equation}
\end{lem}
\begin{proof}
If $\phi:\BB{R},0\to\BB{R}^2,0$ is of multiplicity $m$, then $\frac{d\phi}{dt}$ is divided by $t^{m-1}/(m-1)!$ whose quotient is non-zero $C^\infty$ map. 
That is, there exists a $C^\infty$-map $\BS{T}:\BB{R},0\to\BB{R}^2,0$ so that 
$$
\frac{d\phi}{dt}=\frac{t^{m-1}}{(m-1)!}\BS{T}(t)\quad\text{ with }\quad\BS{T}(0)\ne0. 
$$
Then there is a $C^\infty$-function $\varphi(t)$ such that 
$$
\int_0^t\frac{u^{m-1}}{(m-1)!}|\BS{T}(u)|\,du=\frac{t^m\varphi(t)}{m!}.
$$  
Define a new parameter $s$ by $s=t\varphi(t)^{1/m}$. Then we have 
\begin{equation}\label{E1}
\int_0^t\Bigl|\frac{d\phi}{dt}\Bigr|dt=(\sign t^{m-1})\int_0^t\frac{u^{m-1}}{(m-1)!}|\BS{T}(u)|\,du=(\sign{t^{m-1}})\frac{s^m}{m!}.
\end{equation}
which shows that $\pm s^m/m!$ is an arc length parameter. 

Remark that $d\phi/ds$ is divided by $s^{m-1}/(m-1)!$ and its quotient is non-zero. 
We denote the quotient by $\BS{t}$. 
Differentiating \eqref{E1} by $s$, we obtain that 
\begin{equation}\label{ACparmeter}
\Bigl|\frac{d\phi}{dt}\Bigr|\frac{dt}{ds}=(\sign t^{m-1})\frac{s^{m-1}}{(m-1)!},
\text{ and thus } \ 
\Bigl|\frac{d\phi}{ds}\Bigr|=(\sign s^{m-1})\frac{s^{m-1}}{(m-1)!}.
\end{equation}
This implies that $|\BS{t}|=1$. 
\end{proof}
This is a slight variant of Theorem 1.1 in \cite{Fukui2017}.

We call $s$ the {\bf curvature parameter} of $\phi$. 
In \eqref{ACparmeter}, we assume that $dt/ds$ is positive, which means that 
the orientation of the curve given by the parameter  $s$ agrees with that given by the parameter $t$. 
If $s$ is a solution to the ordinary differential equation 
\begin{equation*}
\frac{s^{m-1}}{(m-1)!}\frac{ds}{dt}=\Bigl|\frac{d\phi}{dt}\Bigr|,
\end{equation*}
then either $s$ or $-s$ is a curvature parameter compatible with the orientation of the curve given by the parameter $t$.   

Let $\BS{e}_1=(1,0)$ and $\BS{e}_2=(0,1)$. 
We assume $\BS{t}|_{s=0}=\BS{e}_1$ without loss of generality.  
We take $\BS{n}:\BB{R},0\longrightarrow\BB{R}^2$ so that $\BS{t}$, $\BS{n}$ form an oriented orthogonal frame of $\BB{R}^2$.
We define the function $\kappa$ with the following differential equation: 
\begin{equation}\label{FS}
\frac{d}{ds}
\begin{pmatrix}
\BS{t}\\
\BS{n}
\end{pmatrix}
=
\begin{pmatrix}
0&\kappa\\
-\kappa&0
\end{pmatrix}
\begin{pmatrix}
\BS{t}\\
\BS{n}
\end{pmatrix},\quad
\binom{\BS{t}}{\BS{n}}\bigg|_{s=0}=\binom{\BS{e}_1}{\BS{e}_2},
\end{equation}
where $\kappa=\kappa(s)$ is a $C^\infty$-function. 

If  a function $\kappa$ is given, then we can define $\BS{t}$ and $\BS{n}$ as a solution to \eqref{FS}, and 
we recover a plane curve $\phi$ integrating \eqref{PT}.
This curve $\phi$ is uniquely determined up to motion, 
when the function $\kappa$ and the multiplicity $m$ are given. 

Setting $\phi_m$ as 
\begin{equation}\label{Def:Phi_m}
\phi_m(s)=\int_0^s\frac{s^{m-1}}{(m-1)!}\,\BS{t}\,ds,
\end{equation}
the curve $\phi$ is congruent to $\phi_m$. 
We remark its curvature is
$$
\frac{\kappa(s)}{s^{m-1}/(m-1)!}.
$$ 
When we have the Taylor expansion of $\BS{t}$ as $\sum_{k\ge0}\BS{t}_k\,s^k/k!$,
the Taylor expansion of $\phi_m$ is given as follows:
\begin{align}\label{Taylor:phi}
\sum_{k=0}^\infty\int\frac{s^{m-1}}{(m-1)!}\,\BS{t}_k\,\frac{s^k}{k!}\,ds
=\sum_{k=0}^\infty\BS{t}_k\,\frac{\binom{m+k-1}{m-1}\,s^{m+k}}{(m+k)!}.
\end{align}

When the Taylor expansion of $\kappa$ is given as $\sum_{i=0}^\infty\kappa_i\,s^i/i!$, 
it determines the Taylor expansion of $\BS{t}$ and thus that of $\phi_m$. 
We will compute several Taylor coefficients of $\phi_m$. 
We first remark that the Taylor expansion of 
$\left(\begin{smallmatrix}
\BS{t}\\ \BS{n}
\end{smallmatrix}\right)$
is that of
\begin{equation}\label{Formula:T}
\exp(\theta J)=
\sum_{j\ge0}\frac{(\theta J)^j}{j!},
\quad \theta=\int_0^s\kappa\,ds, \quad 
J=
\begin{pmatrix}
0&1\\
-1&0
\end{pmatrix}, 
\end{equation}
since this is the solution to \eqref{FS} when $\kappa$ is analytic. 

If the Taylor expansion of $\theta$ is $\sum_{i=1}^\infty\theta_i\,s^i/i!$, then we have $\theta_i=\kappa_{i-1}$.
For the sake of simplicity,
we will work using the coefficients $\theta_i$ instead of $\kappa_i$ in calculation below.
Since 
$J^2=
\left(\begin{smallmatrix}
-1&0\\0&-1
\end{smallmatrix}\right) $,  
$J^3=
\left(\begin{smallmatrix}
0&-1\\1&0
\end{smallmatrix}\right)$, 
$J^4=
\left(\begin{smallmatrix}
1&0\\0&1
\end{smallmatrix}\right)$, 
we have 
\begin{align}\label{Formula:Tk}
\BS{t}_k
=\Bigl(\sum_{i\ge0}(-1)^i\frac{[\theta^{2i}]_k}{(2i)!}\Bigr)\BS{e}_1
+\Bigl(\sum_{i\ge0}(-1)^i\frac{[\theta^{2i+1}]_k}{(2i+1)!}\Bigr)\BS{e}_2
\end{align}
by \eqref{Formula:T} where 
$[\theta^j]_k$ is defined by $\theta^j=\sum_{k\ge0}[\theta^j]_k s^k/k!$, 
that is, 
$$
[\theta^j]_k=k!\sum_{j_1+2j_2+\dots+kj_k=k}
\Bigl[\begin{matrix}{j}\\{j_1\ \dots\ j_k}\end{matrix}\Bigr]
\frac{\theta_1^{j_1}\theta_2^{j_2}\cdots\theta_k^{j_k}}
{1!^{j_1}2!^{j_2}\cdots k!^{j_k}}
$$
where $\left[\begin{smallmatrix}{j}\\{j_1\ \dots\ j_k}\end{smallmatrix}\right]=\frac{j!}{j_1!\cdots j_k!}$, if $j_1+\cdots+j_k=j$; $0$, otherwise.

For the reader's reference, we present the first few terms of the Taylor expansion of $\BS{t}$.
$$
\BS{t}
=
\binom10
+\binom0{\theta_1}s
+\binom{-\theta_1^2}{\theta_2}\frac{s^2}{2!}
+\binom{-3\theta_1\theta_2}{\theta_3-\theta_1^3}\frac{s^3}{3!}
+\binom{\theta_1^4-3\theta_2^2-4\theta_1\theta_3}{\theta_4-6\theta_1^2\theta_2}\frac{s^4}{4!}
+o(s^4).
$$
When $\theta_1=0$, it looks like 
$$
\BS{t}=
\binom10
+\binom0{\theta_2}\frac{s^2}{2!}
+\binom0{\theta_3}\frac{s^3}{3!}
+\binom{-3\theta_2^2}{\theta_4}\frac{s^4}{4!}
+\binom{-10\theta_2\theta_3}{\theta_5}\frac{s^5}{5!}
+\binom{-5(3\theta_2\theta_4+2\theta_3^2)}{\theta_6-15\theta_2^3}\frac{s^6}{6!}
+o(s^6).
$$

In \S\ref{4}, we use the following lemma:
\begin{lem}\label{Lemma:Tk1}
Let us assume that 
\begin{equation}\label{Def:Theta}
\theta=
\theta_m\frac{s^m}{m!}+\theta_{2m}\frac{s^{2m}}{(2m)!}+\cdots+\theta_{lm}\frac{s^{lm}}{(lm)!}+
\theta_p\frac{s^{p}}{p!}+\theta_{p+1}\frac{s^{p+1}}{(p+1)!}+o(s^{p+1}) 
\end{equation}
with $lm<p<(l+1)m$.  Then, we have the following:
\begin{itemize}
\item[\rm (i)] $\BS{t}_0=\BS{e}_1$.
\item[\rm (ii)] $\BS{t}_k=\BS{0}$ if $k\not\equiv0\bmod m$, $1\le k<p$.
\item[\rm (iii)] $\BS{t}_k=\theta_k\,\BS{e}_2$ if $k\not\equiv0\bmod m$, $p\le k<p+m$.
\item[\rm (iv)] 
$\BS{t}_k=\theta_k\,\BS{e}_2$
 if $k\not\equiv0\bmod m$, $p<k<2p$, and 
$\theta_{k-jm}=0$ for $j$ with $0<j<k/m$.
\end{itemize}
\end{lem}
\begin{proof}
The case that $k<p$ is clear. \\
(iii): 
We consider $(j_m,j_{2m},\dots,j_{lm},j_p,\dots,j_k)$ so that 
\begin{align*}
k=&m(j_m+2j_{2m}+\dots+lj_{lm})+pj_p+(p+1)j_{p+1}+\cdots+kj_k.
\end{align*}
If $k\not\equiv 0\bmod m$ and $p\le k<p+m$, such $(j_m,j_{2m},\dots,j_{lm},j_p,\dots,j_k)$ must be $(0,\dots,0,1)$, which implies (iii). 
\\
(iv): 
Assume that $k\not\equiv0\bmod m$ and $p<k<2p$. 
Applying a similar argument to \eqref{Formula:Tk}, we obtain that 
\begin{align*}
\BS{t}_k
=&k!\sum_{r:\text{odd}}
\frac{(-1)^{\frac{r+1}2}}{(r+1)!}
\sum_{j_m+2j_{2m}+\cdots+lj_{lm}=j}
\Bigl[\begin{matrix}r+1\\ {j_m\ j_{2m}\cdots\ j_{lm}\ 1}\end{matrix}\Bigr]
\frac{\theta_{m}^{j_m}\theta_{2m}^{j_{2m}}\cdots\theta_{lm}^{j_{lm}}\theta_{k-jm}}
{m!^{j_m}(2m)!^{j_{2m}}\cdots (lm)!^{j_{lm}}(k-jm)!}
\,\BS{e}_1\\
+&
k!\sum_{r:\text{even}}
\frac{(-1)^{\frac{r}2}}{(r+1)!}
\sum_{j_m+2j_{2m}+\cdots+lj_{lm}=j}
\Bigl[\begin{matrix}r+1\\ {j_m\ j_{2m}\cdots\ j_{lm}\ 1}\end{matrix}\Bigr]
\frac{\theta_{m}^{j_m}\theta_{2m}^{j_{2m}}\cdots\theta_{lm}^{j_{lm}}\theta_{k-jm}}
{m!^{j_m}(2m)!^{j_{2m}}\cdots (lm)!^{j_{lm}}(k-jm)!}\BS{e}_2.
\end{align*}
\hspace{-7pt}
We thus conclude that the coefficient of $\BS{e}_1$ is zero and  
the coefficient of $\BS{e}_2$ is $\theta_k$ if $\theta_{k-jm}=0$ for all $j$ with $0<j<k/m$.
\end{proof}

\begin{rem}
Assume that \eqref{Def:Theta} holds with $l\ge1$, 
we can consider a nonsingular curve
whose Taylor expansion is given by 
\begin{align*}
\sum_{m|k}\int\frac{u^{m-1}}{(m-1)!}\,\BS{t}_k\,\frac{u^{k/m}}{k!}\,du
=\sum_{i=1}^\infty\BS{t}_{m(i-1)}\,\frac{\binom{mi-1}{m-1}\,u^i}{(mi)!},
\end{align*}
as a nonsingular approximation of $\phi_m$. 
Since $\BS{t}_m=\theta_m\BS{e}_2$, the 2-jet of this curve is  
$$
\frac{u}{m!}\BS{e}_1+\frac{\theta_m}2\Bigl(\frac{u}{m!}\Bigr)^2\BS{e}_2.
$$
This implies that the curvature of the nonsingular curve is $\theta_m$ at $u=0$.
Thus, in most cases, $\theta_m$ behaves like the curvature at the singularity; 
that is, $\theta_m^{-1}$ behaves like the radius of curvature --- namely, the parallel curve at this distance has a degenerate singularity whenever $\theta_m\ne0$.  
\end{rem}

\section{Criteria of singularities}\label{3}
Assume that a function $\kappa$ is given, and define a frame $\BS{t}$, $\BS{n}$ by \eqref{FS}.
Let  $\sum_{i\ge0}\kappa_is^i/i!$ denote the Taylor expansion of $\kappa$. 
We define $\phi_m$ by \eqref{Def:Phi_m}.
\begin{thm}\label{Thm:S} 
\begin{itemize}
\item[\rm (i)] The map germ $\phi_2$ defines $A_{2k}$ singularity at $0$ if and only if 
$\kappa_0=\kappa_2=\dots=\kappa_{2k-4}=0$ and $\kappa_{2k-2}\ne0$.
\item[\rm (ii)] The map germ $\phi_3$ defines 
\begin{itemize} 
\item $E_{6k}$ singularity at $0$ if and only if 
$\kappa_{i-1}=0$ $(i\not\equiv0\bmod3, \ i<3k-2)$ and $\kappa_{3k-1}\ne0$.
\item $E_{6k+2}$ singularity at $0$ if and only if  
$\kappa_{i-1}=0$ $(i\not\equiv0\bmod3, \ i<3k-1)$ and $\kappa_{3k-2}\ne0$.  
\end{itemize}
\item[\rm (iii)] The map germ $\phi_4$ defines 
\begin{itemize}
\item $W_{12}$ singularity at $0$ if and only if $\kappa_0\ne0$.
\item $W^{\#}_{1,2q-1}$ singularity at $0$ if and only if 
$\kappa_0=0$, $ \kappa_1\ne0$, $\kappa_2=\kappa_4=\cdots=\kappa_{2q-2}=0$ and $\kappa_{2q}\ne0.$
\item $W_{18}$ singularity at $0$ if and only if 
$\kappa_0=\kappa_1=0$ and $\kappa_2\ne0. $
\end{itemize}
\end{itemize}
\end{thm}

It is more convenient to state our criteria of singularities for $\mathcal A$-simple singularities 
in terms of the Taylor coefficients of $\theta=\int_0^s\kappa\,ds$ (see \eqref{Formula:T}). 
Let $\sum_{i\ge1}\theta_is^i/i!$ denote the Taylor expansion of $\theta$.
Our criteria given below imply Theorem \ref{Thm:S} immediately, since $\kappa_{i-1}=\theta_i$. 
\begin{thm}\label{Thm:M2}
The map germ $\phi_2$ defines $A_{2k}$ singularity at $0$ if and only if 
\begin{equation}\label{Condition:Ak}
\theta_i=0,\quad i\not\equiv0\bmod2,\ i<2k-1;\qquad \theta_{2k-1}\ne0. 
\end{equation}
If these conditions hold, $\phi_2$ is $\mathcal L$-equivalent to $(t^2,t^{2k+1})$ at $0$.  
\end{thm}
The following table summarises our criteria for the $A_{2k}$ singularity with $k\le4$.
\begin{center}
\begin{tabular}{c|l|l}
\hline
$m=2$&normal form &condition\\
\hline
$A_2$
&$(t^2,t^3)$
&$\theta_1\ne0$\\
\hline
$A_4$
&$(t^2,t^5)$
&$\theta_1=0$, $\theta_3\ne0$\\
\hline
$A_6$
&$(t^2,t^7)$
&$\theta_1=\theta_3=0$, $\theta_5\ne0$\\
\hline
$A_8$
&$(t^2,t^9)$
&$\theta_1=\theta_3=\theta_5=0$, $\theta_7\ne0$\\
\hline
\end{tabular}
\end{center}
\begin{thm}\label{Thm:M3}
\begin{itemize}
\item[\rm (i)] 
The map germ $\phi_3$ defines $E_{6k}$ singularity at $0$ if and only if 
\begin{equation}\label{Condition:E_6k}
\theta_i=0,\ i\not\equiv0\bmod3,\ i< 3k-2; \qquad\theta_{3k-2}\ne0.
\end{equation}
Actually under the assumption \eqref{Condition:E_6k}, 
$\phi_3$ is $\mathcal A$-equivalent to 
\begin{itemize}
\item[\rm(ia)] $(t^3,t^{3k+1}+\varepsilon_pt^{3(k+p)+2})$ at $0$ if 
\begin{equation}\label{ia}
\theta_{3(k+j)-1}=0,\ 0\le j<p; \quad \theta_{3(k+p)-1}\ne0.
\end{equation}
Moreover, $\varepsilon_p$ is the sign of $\tfrac{\theta_{3(k+p)-1}}{\theta_{3k-2}}$ when $p$ is odd.
Here $p$ is an integer with $0\le p \le k-2$;
\item[\rm(ib)] $(t^3,t^{3k+1})$ if 
\begin{equation}\label{ib}
\theta_{3(k+j)-1}=0,\ 0\le j\le k-2.
\end{equation}
\end{itemize}
\item[\rm (ii)] 
The map $\phi_3$ defines $E_{6k+2}$ singularity at $0$ if and only if 
\begin{equation}\label{Condition:E_6k+2}
\theta_i=0,\ i\not\equiv0\bmod3,\ i< 3k-1;\quad\theta_{3k-1}\ne0.
\end{equation}
Actually under the assumption \eqref{Condition:E_6k+2}, 
$\phi_3$ is $\mathcal A$-equivalent to 
\begin{itemize}
\item[\rm (iia)] $(t^3,t^{3k+2}+\varepsilon_{p+1}t^{3(k+p)+4})$ if 
\begin{equation}\label{iia}
\theta_{3(k+j)+1}=0, \ 0\le j<p, \ \text{ and } \ \theta_{3(k+p)+1}\ne0,
\end{equation} 
Moreover, $\varepsilon_{p+1}$ is the sign of $\tfrac{\theta_{3(k+p)+1}}{\theta_{3k-1}}$ when $p+1$ is odd.
Here $p$ is an integer with $0\le p\le k-2$; 
\item[\rm (iib)] $(t^3,t^{3k+2})$ if  
\begin{equation}\label{iib}
\theta_{3(k+j)+1}=0, \quad 0\le j\le k-2.
\end{equation}
\end{itemize}
\end{itemize}
\end{thm}
The following table summarises our criteria for the $E$-series singularities. 
\begin{center}
\begin{tabular}{c|l|l}
\hline
$m=3$&normal form &condition\\
\hline
$E_6$
&$(t^3,t^4)$&$\theta_1\ne0$\\
\hline
$E_8$
&$(t^3,t^5)$&$\theta_1=0$, $\theta_2\ne0$\\
\hline
$E_{12}$
&$(t^3,t^7+t^8)$&$\theta_1=\theta_2=0$, $\theta_4\ne0$, $\theta_5\ne0$\\
&$(t^3,t^7)$&$\theta_1=\theta_2=0$, $\theta_4\ne0$, $\theta_5=0$\\
\hline
$E_{14}$
&$(t^3,t^8\pm t^{10})$
&$\theta_1=\theta_2=\theta_4=0$, $\theta_5\ne0$, $\pm\theta_7/\theta_5>0$\\
&$(t^3,t^8)$
&$\theta_1=\theta_2=\theta_4=0$, $\theta_5\ne0$, $\theta_7=0$\\
\hline
$E_{18}$
&$(t^3,t^{10}+t^{11})$&$\theta_1=\theta_2=\theta_4=\theta_5=0$, $\theta_7\ne0$, $\theta_8\ne0$\\
&$(t^3,t^{10}\pm t^{14})$&
$\theta_1=\theta_2=\theta_4=\theta_5=0$, $\theta_7\ne0$, $\theta_8=0$, $\pm\theta_{11}/\theta_7>0$\\
&$(t^3,t^{10})$&
$\theta_1=\theta_2=\theta_4=\theta_5=0$, $\theta_7\ne0$, $\theta_8=\theta_{11}=0$\\
\hline
$E_{20}$
&$(t^3,t^{11}\pm t^{13})$
&$\theta_1=\theta_2=\theta_4=\theta_5=\theta_7=0$, $\theta_8\ne0$, $\pm\theta_{10}/\theta_8>0$\\
&$(t^3,t^{11}+t^{16})$&
$\theta_1=\theta_2=\theta_4=\theta_5=\theta_7=0$, $\theta_8\ne0$, $\theta_{10}=0$, $\theta_{13}\ne0$\\
&$(t^3,t^{11})$&
$\theta_1=\theta_2=\theta_4=\theta_5=\theta_7=0$, $\theta_8\ne0$, $\theta_{10}=\theta_{13}=0$\\
\hline
\end{tabular}
\end{center}

\begin{thm}\label{Thm:M4}
\begin{itemize}
\item[\rm (i)] 
The map germ $\phi_4$ defines $W_{12}$ singularity at $0$ if and only if 
$$
\theta_1\ne0.
$$
Moreover, $\phi_4$ is $\mathcal A$-equivalent to $(t^4,t^5\pm t^7)$ $(\textrm{resp.}~(t^4,t^5))$ at $0$ 
if $\pm w_{12}>0$ $(\textrm{resp.}~w_{12}=0)$
where 
$w_{12}=\theta_1\theta_3-\tfrac{77}{48}\,\theta_2^2+\tfrac52\,\theta_1^4$.
\item[\rm (ii)] 
The map germ $\phi_4$ defines $W^\#_{1,2q-1}$ singularity at $0$ if and only if 
\begin{equation}\label{WC2}
\theta_1=0, \quad \theta_2\ne0, \quad\theta_3=\theta_5=\cdots=\theta_{2q-1}=0 \text{ and } \theta_{2q+1}\ne0.
\end{equation}
\item[\rm (iii)] 
The map germ $\phi_4$ defines $W_{18}$ singularity at $0$ if and only if 
$$
\theta_1=\theta_2=0\ \text{ and }\ \theta_3\ne0.
$$
Moreover, $\phi_4$ is $\mathcal A$-equivalent to  
\begin{itemize}
\item $(t^4,t^7\pm t^9)$ if $\pm\tfrac{\theta_5}{\theta_3}>0$,
\item $(t^4,t^7\pm t^{13})$ if $\theta_5=0\text{ and }\pm w_{18}>0$,
\item $(t^4,t^7)$ if $\theta_5=w_{18}=0$, 
\end{itemize}
where $w_{18}=\theta_3\theta_9-\tfrac{4641}{1000}\,\theta_6^2+812\,\theta_3^4$. 
\end{itemize}
\end{thm}
The following table summarises our criteria for $m=4$. 
\begin{center}
\begin{tabular}{c|l|l}
\hline 
$m=4$&normal form &condition\\
\hline
$W_{12}$
&$(t^4,t^5\pm t^7)$
&$\theta_1\ne0$, $\pm w_{12}>0$\\
&$(t^4,t^5)$
&$\theta_1\ne0$, $w_{12}=0$\\
\hline
$W^\#_{1,1}$
&$(t^4,t^6+t^7)$
&$\theta_1=0$, $\theta_2\ne0$, $\theta_3\ne0$\\
\hline
$W^\#_{1,3}$
&$(t^4,t^6+t^9)$
&$\theta_1=0$, $\theta_2\ne0$, $\theta_3=0$, $\theta_5\ne0$\\
\hline
$W^\#_{1,5}$
&$(t^4,t^6+t^{11})$
&$\theta_1=0$, $\theta_2\ne0$, $\theta_3=\theta_5=0$, $\theta_7\ne0$\\
\hline
$W_{18}$
&$(t^4,t^7\pm t^9)$
&$\theta_1=\theta_2=0$, $\theta_3\ne0$, $\pm\theta_5/\theta_3>0$\\
&$(t^4,t^7\pm t^{13})$
&$\theta_1=\theta_2=0$, $\theta_3\ne0$, $\theta_5=0$, $\pm w_{18}>0$\\
&$(t^4,t^7)$
&$\theta_1=\theta_2=0$, $\theta_3\ne0$, $\theta_5=0$, $w_{18}=0$\\
\hline
\end{tabular}
\end{center}
\begin{rem}
Matsushita (\cite[Theorem 4.14]{M}) gives a similar criterion for $W_{12}$ singularity. 
\end{rem}

\section{Proofs of criteria}\label{4}
Since $\mathcal A$-simple singularities are finitely determined, the proofs of Theorems \ref{Thm:M2}, \ref{Thm:M3} and \ref{Thm:M4}  are obtained by aligning the coefficients of the Taylor expansion of the given map with those of the normal form up to sufficiently high orders, via  suitable coordinate changes of the source and the target. 
For the definition and basic properties of finite determinacy, the reader may consult \cite{Wall}.

The process of aligning coefficients via coordinate changes is based on the following lemma. 

\begin{lem}\label{Lemma:Key}
Assume that $n\gg 1$ and $\phi:\BB{R},0\longrightarrow\BB{R}^2,0$ is expressed as  
\begin{equation}\label{Def:Lemma41}
\phi(t)=(x(t),y(t))=\Bigl(\sum_{i=m}^na_it^{i}+o(t^n),\ b_{m+p}t^{m+p}+\sum_{j=j_1}^nb_jt^j+o(t^n)\Bigr)
\end{equation}
with $\ a_{m}\ne0$, $b_{m+p}\ne0$, $p\not\equiv0\bmod m$, $b_{j_1}\ne0$.
We assume that $m+p<j_1$.
For a diffeomorphism
$\Phi:\BB{R}^2,0\to\BB{R}^2,0$ so that 
$$
\Phi(x,y)=\Bigl(
\sum_{1\le i+j\le n}p_{i,j}x^iy^j+o(|(x,y)|^n),\ 
\sum_{1\le i+j\le n}q_{i,j}x^iy^j+o(|(x,y)|^n)
\Bigr), 
$$ 
and a difffeomorphism $h:\BB{R},0\to\BB{R},0$ so that  
$$
h(t)=\sum_{k=1}^nh_k\frac{t^k}{k!}+o(t^n),
$$
we set $\tilde\phi(t)=\Phi\comp\phi\comp h(t)$ and express it as 
$$
\tilde\phi(t)=(\tilde{x}(t),\tilde{y}(t))
=\Bigl(\sum_{i=m}^n\tilde{a}_it^i+o(t^n),\sum_{j=m}^n\tilde{b}_jt^j+o(t^n)\Bigr).
$$
Then $\tilde{a}_i$ and $\tilde{b}_j$ are expressed as polynomials of $a_i$, $b_j$, $p_{i,j}$, $q_{i,j}$ and $h_k$. 
In particular, the polynomials $\tilde{a}_i$ and $\tilde{b}_j$ are linear in $p_{i,j}$, $q_{i,j}$. 
Moreover, we have the following: 
\begin{itemize}
\item[{\rm (i)}] 
If $e=im+j(m+p)$ for some $i,j\in\BB{Z}_\ge$,  
then the polynomial $\tilde{a}_e$ $(\textrm{resp.}~\tilde{b}_e)$ 
contains the term $p_{i,j}a_{m}^ib_{m+p}^j$ $(\textrm{resp.}~q_{i,j}a_{m}^ib_{m+p}^j)$. 
\item[{\rm (ii)}] 
If $e=im+j_1$ for some $i\in\BB{Z}_\ge$, then the polynomial $\tilde{a}_e$ $(\textrm{resp.}~\tilde{b}_e)$ 
contains the term $p_{i,1}a_{m}^ib_{j_1}$ $(\textrm{resp.}~q_{i,1}a_{m}^ib_{j_1})$. 
\item[\rm (iii)] 
For $c\ge1$, the polynomial $\tilde{a}_{m+c}$ $(\textrm{resp.}~\tilde{b}_{m+p+c})$ contains the term 
$$
a_{m}h_1^{m-1}h_{c+1}\quad(\textrm{resp.}~b_{m+p}h_1^{m+p-1}h_{c+1}),
$$ 
and the polynomials $\tilde{a}_{m+i}$ $(\textrm{resp.}~\tilde{b}_{m+p+i})$, $0\le i<c$, does not contains  
$h_{c+1}$.
\end{itemize}
\end{lem}
\begin{proof}
Trivial.
\end{proof}

When $p\equiv 0\bmod{m}$ in \eqref{Def:Lemma41}, there is $r$  with $p=mr$, and 
we can eliminate the term $t^{m+p}$ by changing $(x,y)$ by $(x,y-cx^r)$ with some constant $c$. 

Now we prove Theorem \ref{Thm:M2}.
\begin{proof}[Proof of Theorem \ref{Thm:M2}]
It is enough to show that $\phi_2$ is $\mathcal L$-equivalent to $(t^2,t^{2k+1})$, 
up to sufficiently high orders, if 
\eqref{Condition:Ak} holds.
We assume \eqref{Condition:Ak}. Then, by \eqref{Taylor:phi}, we have 
$$
\phi_2(s)=\BS{e}_1\frac{s^2}{2}+\sum_{i=2}^k\BS{t}_{2i-2}\frac{(2i-1)s^{2i}}{(2i)!}
+\theta_{2k-1}\BS{e}_2\frac{2k\,s^{2k+1}}{(2k+1)!}+o(s^{2k+1}).
$$
Setting $S=\{2i+(2k+1)j:i,j\in\BB{Z}_{\ge}\}$, we have
$$
S=\{e\in\BB{Z}_{\ge}:e\equiv 0\bmod2\}\cup\{e\in\BB{Z}_{\ge}:e\ge2k\}.
$$
By Lemma \ref{Lemma:Key}, we have the following:
\begin{itemize}
\item For $e\in S$, we can send the coefficient of $\BS{e}_1s^{e}$ to $0$ ($1$ when $e=2$) 
choosing $p_{i,j}$ suitably for $(i,j)$ with $e=2i+(2k+1)j$ (by Lemma \ref{Lemma:Key}, (i)). 
\item For $e\in S$, we can send the coefficient of $\BS{e}_2s^{e}$ to $0$ ($1$ when $e=2k+1$) 
choosing $q_{i,j}$ suitably for $(i,j)$ with $e=2i+(2k+1)j$ (by Lemma \ref{Lemma:Key}, (i)).
\end{itemize}
Applying these processes repeatedly, we easily complete the proof.
\end{proof}

We next prove Theorem \ref{Thm:M3}.

\begin{proof}[Proof of Theorem \ref{Thm:M3}]
The proof is based on the repeated application of Lemma \ref{Lemma:Key}. 

(i): We first assume \eqref{Condition:E_6k}. 
Then, by \eqref{Taylor:phi}, we have
\begin{equation}\label{Proof:E_6k}
\phi_3(s)=\BS{e}_1\frac{s^3}{6}
+\sum_{i=2}^{k}\BS{t}_{3(i-1)}\frac{\binom{3i-1}2s^{3i}}{(3i)!}+\theta_{3k-2}\BS{e}_2\frac{\binom{3k}2s^{3k+1}}{(3k+1)!}+o(s^{3k+1}).
\end{equation}
We choose $h_{i+1}=0$ for $i$ with $i\not\equiv0\bmod3$, $1\le i<3k-2$, in order not to change the coefficient of $\BS{e}_1s^{i+3}$.

Setting $S=\{3i+(3k+1)j:i,j\in\BB{Z}_{\ge}, i+j\ge1\}$, we have 
$$
S=\{e\in\BB{Z}_{\ge}: e\not\equiv 2\bmod3\}\cup\{e\in\BB{Z}: e\ge6k\}.
$$
By Lemma \ref{Lemma:Key}, we have the following: 
\begin{itemize}
\item
For $e\in S$, we can send the coefficients of $\BS{e}_1s^e$ to $0$ (resp.~$1$), when $e\ne3$ (resp.~$e=3$), choosing $p_{i,j}$  (resp.~$p_{1,0}$) suitably  (by Lemma \ref{Lemma:Key}, (i)).
\item
For $e\in S$, we can send the coefficients of $\BS{e}_2s^e$ to $0$ (resp.~$1$), when $e\ne3k+1$ (resp.~$e=3k+1$), choosing $q_{i,j}$  (resp.~$q_{0,1}$) suitably (by Lemma \ref{Lemma:Key}, (i)).
\item
For $q=0,1,2,\dots,k-1$, we can send the coefficient of $\BS{e}_1s^{3(k+q)+2}$ to $0$, choosing 
$h_{3(k+q)}$ suitably (by Lemma \ref{Lemma:Key}, (iii)).
\item We can send the coefficient of $\BS{e}_2s^{6k-1}$ to $0$, choosing $h_{3k-1}$ suitably (by Lemma \ref{Lemma:Key}, (iii)).
\item 
When \eqref{ia} holds, by Lemma \ref{Lemma:Tk1} (iv), 
the coefficient of $\BS{e}_2s^{3(k+q)+2}$ in \eqref{Proof:E_6k} is 
$$
\begin{cases}
0&(q=0,1,\dots,p-1),\\
\frac{\binom{3(k+p)-1}{2}}{(3(k+p)+2)!}\theta_{3(k+p)-1}&(q=p).
\end{cases}
$$
\item We can send the coefficient of $\BS{e}_2s^{3(k+p)+2}$ to $\pm1$, choosing $h_1$ suitably.
\item For $q$ with $1\le q<k-p$, 
we can send the coefficient of $\BS{e}_2s^{3(k+p+q)+2}$ to $0$, choosing $h_{3(p+q)+2}$ suitably.
\end{itemize}

When \eqref{ib} holds, we repeat these operations to align the coefficients in order from the lowest-order terms, 
and we complete the proof. 

The processes to determine $p_{i,j}$, $q_{i,j}$ and $h_i$, for the cases $k\le4$, are summarized as the following tables:
\begin{center}
\begin{tabular}{|c|c|c|c|c|c|c|c|}
\hline
$E_6$&$s^3$&$s^4$&$s^5$&$s^6$&$s^7$\\
\hline
$\BS{e}_1$&$p_{10}$&$p_{01}$&$h_3$&$p_{20}$&$p_{11}$\\
$\BS{e}_2$&&$q_{01}$&$h_2$&$q_{20}$&$q_{11}$\\
\hline
\end{tabular}
\end{center}
\begin{center}
\begin{tabular}{|c|c|c|c|c|c|c|c|c|c|c|c|c|c|c|}
\hline
$E_{12}$&$s^3$&$s^4$&$s^5$&$s^6$&$s^7$&$s^8$&$s^9$&$s^{10}$&$s^{11}$&$s^{12}$&$s^{13}$\\
\hline
$\BS{e}_1$&$p_{10}$&$h_2$&$h_3$&$p_{20}$&$p_{01}$&$h_6$&$p_{30}$&$p_{11}$&$h_9$&$p_{40}$&$p_{31}$\\
$\BS{e}_2$&&&&$q_{20}$&$q_{01}$&&$q_{30}$&$q_{11}$&$h_5$&$q_{40}$&$q_{31}$\\
\hline
\end{tabular}
\end{center}
\begin{center}
\scriptsize
\begin{tabular}{|c|c|c|c|c|c|c|c|c|c|c|c|c|c|c|c|c|c|}
\hline
$E_{18}$&$s^3$&$s^4$&$s^5$&$s^6$&$s^7$&$s^8$&$s^9$&$s^{10}$&$s^{11}$&$s^{12}$&$s^{13}$&$s^{14}$&$s^{15}$&$s^{16}$&$s^{17}$&$s^{18}$&$s^{19}$\\
\hline
$\BS{e}_1$&$p_{10}$&$h_2$&$h_3$&$p_{20}$&$h_5$&$h_6$&$p_{30}$&$p_{01}$&$h_9$&$p_{40}$&$p_{11}$&$h_{12}$&$p_{50}$&$p_{21}$&$h_{15}$&$p_{60}$&$p_{31}$\\
$\BS{e}_2$&&&&$q_{20}$&&&$q_{30}$&$q_{01}$&&$q_{40}$&$q_{11}$&&$q_{50}$&$q_{21}$&$h_8$&$q_{60}$&$q_{31}$\\
\hline
\end{tabular}
\end{center}
\begin{center}
\includegraphics[width=\textwidth,bb=0 0 550 37]{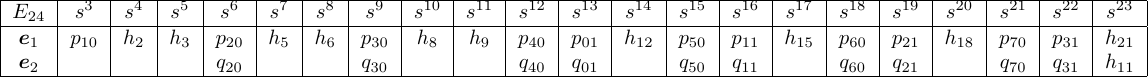}
\end{center}
(ii): Next we assume \eqref{Condition:E_6k+2}. 
Then, by \eqref{Taylor:phi}, we have 
\begin{equation}\label{Proof:E_6k+2}
\phi_3(s)=\BS{e}_1\frac{s^3}{6}
+\sum_{i=2}^k\BS{t}_{3(i-1)}\frac{\binom{3i-1}2s^{3i}}{(3i)!}
+\theta_{3k-1}\BS{e}_2\frac{\binom{3k+1}2s^{3k+2}}{(3k+2)!}+o(s^{3k+2}).
\end{equation}
We choose $h_{i+1}=0$ for $i$ with $i\not\equiv0\bmod3$, 
$1\le i<3k-1$, in order not to change the coefficients of $s^{i+3}\BS{e}_1$.

Setting $S=\{3i+(3k+2)j:i,j\in\BB{Z}_{\ge}, i+j\ge1\}$, we have 
$$
S=\{e\in\BB{Z}_{\ge}: e\not\equiv 1\bmod3\}\cup\{e\in\BB{Z}: e\ge6k+2\}.
$$
By Lemma \ref{Lemma:Key}, we have the following:
\begin{itemize}
\item
For $e\in S$, we can send the coefficients of $\BS{e}_1s^e$ to $0$ (resp.~$1$), when $e\ne3$ (resp.~$e=3$), choosing $p_{i,j}$ (resp.~$p_{1,0}$) suitably (by Lemma \ref{Lemma:Key}, (i)).
\item
For $e\in S$, we can send the coefficients of $\BS{e}_2s^e$ to $0$ (resp.~$1$), when $e\ne3k+2$ (resp.~$e=3k+2$), choosing $q_{i,j}$ (resp.~$q_{0,1}$) suitably (by Lemma \ref{Lemma:Key}, (i)).
\item
For $q=1,2,\dots,k$, we can send the coefficient of $\BS{e}_1s^{3(k+q)+1}$ to zero, choosing $h_{3(k+q)-1}$ suitably (by Lemma \ref{Lemma:Key}, (iii)).
\item We can send the coefficient of $\BS{e}_2s^{6k+1}$ to zero, choosing $h_{3k+1}$ suitably (by Lemma \ref{Lemma:Key}, (iii)).
\item 
When \eqref{iia} holds, by Lemma \ref{Lemma:Tk1} (iv), the coefficient of $\BS{e}_2s^{3(k+q)+1}$ in \eqref{Proof:E_6k+2} is 
$$
\begin{cases}
0&(q=1,2,\dots,p-1),\\
\frac{\binom{3(k+p)+1}{2}}{(3(k+p)+4)!}\theta_{3(k+p)+1}&(q=p).
\end{cases}
$$
\item 
We can send the coefficient of $\BS{e}_2s^{3(k+p)+2}$ is $\pm1$, choosing $h_1$ suitably.
\item 
For $q$ with $1\le q<k-p$, 
we can send the coefficient of $\BS{e}_2s^{3(k+p+q)+1}$ to zero, choosing $h_{3{p+q}}$  suitably.  
\end{itemize}

When \eqref{iib} holds, we repeat these operations to align the coefficients in order from the lowest-order terms, 
and we complete the proof.

The processes to determine $p_{i,j}$, $q_{i,j}$ and $h_i$, for the cases $k\le4$,  are summarized as the following tables:
\begin{center}
\noindent
\begin{tabular}{|c|c|c|c|c|c|c|c|c|c|c|c|c|c|c|}
\hline
$E_8$&$s^3$&$s^4$&$s^5$&$s^6$&$s^7$&$s^8$&$s^9$&$s^{10}$&$s^{11}$&$s^{12}$\\
\hline
$\BS{e}_1$&$p_{10}$&$h_2$&$q_{01}$&$p_{20}$&$h_5$&$p_{11}$&$p_{11}$&$p_{02}$&$p_{21}$&$p_{40}$\\
$\BS{e}_2$&&&$q_{01}$&$q_{20}$&$h_3$&$q_{11}$&$q_{11}$&$q_{02}$&$q_{21}$&$q_{40}$\\
\hline
\end{tabular}
\end{center}
\begin{center}
\begin{tabular}{|c|c|c|c|c|c|c|c|c|c|c|c|c|c|c|}
\hline
$E_{14}$&$s^3$&$s^4$&$s^5$&$s^6$&$s^7$&$s^8$&$s^9$&$s^{10}$&$s^{11}$&$s^{12}$&$s^{13}$&$s^{14}$&$s^{15}$&$s^{16}$\\
\hline
$\BS{e}_1$&$p_{10}$&$h_2$&$h_3$&$p_{20}$&$h_5$&$p_{01}$&$p_{30}$&$h_8$&$p_{11}$&$p_{40}$&$h_{11}$&$p_{21}$&$p_{50}$&$p_{02}$\\
$\BS{e}_2$&&&&$q_{20}$&&$q_{01}$&$q_{30}$&&$q_{11}$&$q_{40}$&$h_6$&$q_{21}$&$q_{50}$&$q_{02}$\\
\hline
\end{tabular}
\end{center}
\begin{center}
\includegraphics[width=\textwidth,bb=0 0 473 37]{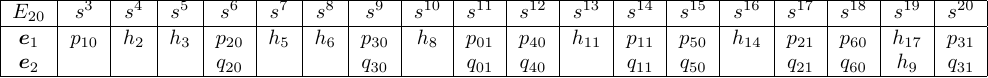}
\end{center}
\begin{center}
\includegraphics[width=\textwidth,bb=0 0 602 37]{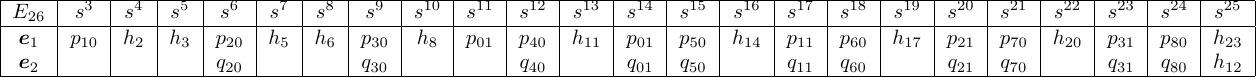}
\end{center}
\end{proof}

\begin{proof}[Proof of Theorem \ref{Thm:M4}]
The proof is based on the repeated application of Lemma \ref{Lemma:Key}. 

(i) We first assume that $\theta_1\ne0$.
Then, by \eqref{Taylor:phi} and \eqref{Formula:Tk}, we have 
\begin{equation*}
\phi_4(s)
=\BS{e}_1\frac{s^4}{4!}
+\theta_1\BS{e}_2\frac{\binom{4}3s^5}{5!}
+(-\theta_1^2\BS{e}_1+\theta_2\BS{e}_2)\frac{\binom{5}3s^6}{6!}
+(-3\theta_2\theta_2\BS{e}_1+(\theta_3-\theta_1^3)\BS{e}_2)\frac{\binom{6}3s^7}{7!}
+o(s^7).
\end{equation*}
Setting $S=\{4i+5j: i,j\in\BB{Z}_{\ge},i+j\ge1\}$, we obtain
$$
S=\{4, 5, 8, 9, 10\}\cup\{e\in\BB{Z}: e\ge12\}. 
$$
By Lemma \ref{Lemma:Key}, we have the following:
\begin{itemize}
\item
For $e\in S$, we can choose $p_{i,j}$  to send the coefficients of $\BS{e}_1s^e$ to  $0$ (resp.~$1$) if $e\ne4$ (resp.~$e=4$) (by Lemma \ref{Lemma:Key}, (i)).
\item
For $e\in S$, we can choose $q_{i,j}$  to send the coefficients of $\BS{e}_2s^e$ to  $0$ (resp.~$1$) if $e\ne5$ (resp.~$e=5$) (by  Lemma \ref{Lemma:Key}, (i)).
\item
We can choose $h_2$ and $h_3$, to send the coefficients of $s^6$ to zero (by  Lemma \ref{Lemma:Key}, (iii)).
\item
We can choose $h_4$, to send the coefficients of $\BS{e}_1s^7$ to zero 
(by Lemma \ref{Lemma:Key}, (iii)).
\item
We can choose $h_7$ and $h_8$, to send the coefficients of $s^{11}$ to zero
(by Lemma \ref{Lemma:Key}, (iii)).
\end{itemize}

We remark that the coefficient of  $\BS{e}_2s^7$ becomes $5h_1^2w_{12}$, when we align the coefficients.
We can choose $h_1$ to send the coefficient of  $\BS{e}_2s^7$ to $\pm1$ when $w_{12}$ is not zero.

Applying these processes repeatedly, we complete the proof. 
The processes to determine $p_{i,j}$, $q_{i,j}$ and $h_k$ are summarized as the following table:
\begin{center}
\begin{tabular}{|c|l|c|c|c|c|c|c|c|c|c|c|c|c|c|c|}
\hline
$W_{12}$&$s^4$&$s^5$&$s^6$&$s^7$&$s^8$&$s^9$&$s^{10}$&$s^{11}$&$s^{12}$\\
\hline
$\BS{e}_1$&$p_{10}$&$p_{01}$&$h_3$&$h_4$&$p_{20}$&$p_{11}$&$p_{02}$&$h_8$&$p_{30}$\\
$\BS{e}_2$&&$q_{01}$&$h_2$&&$q_{20}$&$q_{11}$&$q_{02}$&$h_7$&$q_{30}$\\
\hline
\end{tabular}
\end{center}
See Remark \ref{remW12}, also.

(ii) We assume that $\theta_1=0$ and $\theta_2\ne0$.
Then, by \eqref{Taylor:phi}, we have 
\begin{equation*}
\phi_4(s)=\BS{e}_1\frac{s^4}{4!}+\theta_2\BS{e}_2\frac{\binom{5}3s^6}{6!}+o(s^6).
\end{equation*}
We assume \eqref{WC2}. Then, by \eqref{Taylor:phi} and Lemma \ref{Lemma:Tk1}, , we have 
\begin{equation*}
\phi_4(s)=\BS{e}_1\frac{s^4}{4!}+\theta_2\BS{e}_2\frac{\binom{5}3s^6}{6!}
+\sum_{i=2}^q\BS{t}_{2i}\frac{\binom{2i+3}3s^{2i+4}}{(2i+4)!}
+\theta_{2q+1}\BS{e}_2\frac{\binom{2q+4}3s^{2q+5}}{(2q+5)!}+o(s^{2q+5}).
\end{equation*}
Setting $S_0=\{4i+6j: i,j\in\BB{Z}_{\ge},i+j\ge1, \ j\ne1\}$, we have 
$$
S_0=\{4,8,12\}\cup\{i\in\BB{Z}: i\equiv0\bmod2,\ i\ge16\}.
$$
By Lemma \ref{Lemma:Tk1} (iv), we have that the coefficient of $\BS{e}_2$ in $\BS{t}_{2q+1}$ is $\theta_{2q+1}$, which is non zero, thus 
the coefficient of  $\BS{e}_2s^{2q+5}$ in $\phi_4$ is not zero. 
Setting $S_1=\{4i+(2q+5): i\in\BB{Z}_{\ge}\}$, we have 
$$
S_1=\{e\in\BB{Z}: e\equiv2q+1\bmod4,\ e\ge2q+9\}.
$$
By Lemma \ref{Lemma:Key}, we have the following:
\begin{itemize}
\item 
For $e\in S_0$, we can choose $p_{i,j}$, $j\ne1$, to send the coefficients of $\BS{e}_1s^e$ to  $0$ (resp.~$1$) if $e\ne4$ (resp.~$e=4$) (by Lemma \ref{Lemma:Key}, (i)). 
\item
For $e\in S_0$, we can choose $q_{i,j}$, $j\ne1$, to send the coefficients of $\BS{e}_2s^e$ to $0$ (by Lemma \ref{Lemma:Key}, (i)).
\item 
We can choose $q_{0,1}$, to send the coefficients of $\BS{e}_2s^6$ to $1$ (by Lemma \ref{Lemma:Key}, (i)).
\item
We can choose $h_3$ to send the coefficient of $\BS{e}_1s^6$ to $0$ (by Lemma \ref{Lemma:Key}, (iii)).
\item
We can choose $h_5$ and $h_7$ to send the coefficient of $s^{10}$ to $0$ (by Lemma \ref{Lemma:Key}, (iii)).
\item
We can choose $h_9$ and $h_{11}$ to send the coefficient of $s^{14}$ to $0$ (by Lemma \ref{Lemma:Key}, (iii)).
\item 
For $e\in S_1$, we can choose $p_{i,1}$  to send the coefficients of $\BS{e}_1s^e$ to  $0$ (by Lemma \ref{Lemma:Key}, (ii)).
\item 
For $e\in S_1$, we can choose $q_{i,1}$, $i\ne0$,  to send the coefficients of $\BS{e}_2s^e$ to $0$ if $e\ne2q+5$ (by Lemma \ref{Lemma:Key}, (ii)).
\item
For $e$ with $e\equiv 2q+3\bmod4$, $e>4$, the coefficient of $\BS{e}_1s^e$ become zero choosing $h_j$ suitably where $j=(2q+7+4i)-3=2q+4i+4$ (by Lemma \ref{Lemma:Key}, (iii)).
\item
For $e$ with $e\equiv 2q+3\bmod4$, $e>6$, 
the coefficient of $\BS{e}_2s^e$ becomes zero choosing $h_j$ suitably where 
$j=(2q+7+4i)-5=2q+4i+2$ (by Lemma \ref{Lemma:Key}, (iii)).
\item 
We can choose $h_1$ to send the coefficient of $\BS{e}_2s^{2q+5}$ to $1$.
\end{itemize}
 
Based on the above, it is possible to organize an induction to align the coefficients so that they are in normal form. In fact, first, the odd-order coefficients should be aligned to a sufficiently high order, and then the even-order coefficients should be aligned. 
By repeating this operation, the map under consideration can be reduced to its normal form. 

The processes to determine $p_{i,j}$, $q_{i,j}$ and $h_k$, for the cases $q\le3$,  are summarized as the following tables:
\begin{center}
\begin{tabular}{|c|c|c|c|c|c|c|c|c|c|c|c|c|c|c|c|}
\hline
$W^\#_{1,1}$
&$s^4$&$s^5$&$s^6$&$s^7$&$s^8$&$s^9$&$s^{10}$&$s^{11}$&$s^{12}$&$s^{13}$&$s^{14}$&$s^{15}$&$s^{16}$&$s^{17}$&$s^{18}$\\
\hline
$\BS{e}_1$&$p_{10}$&$h_2$&$p_{01}$&$h_4$&$p_{20}$&$h_6$&$p_{11}$&$h_8$&$p_{30}$&$h_{10}$&$p_{21}$&$h_{12}$&$p_{40}$&$h_{14}$&$p_{03}$\\
$\BS{e}_2$&&&$q_{01}$&$h_1$&$q_{20}$&$h_3$&$q_{11}$&$h_5$&$q_{30}$&$h_7$&$q_{21}$&$h_9$&$q_{40}$&$h_{11}$&$q_{03}$\\
\hline
\end{tabular}
\end{center}
\begin{center}
\begin{tabular}{|c|c|c|c|c|c|c|c|c|c|c|c|c|c|c|c|}
\hline
$W^\#_{1,3}$&$s^4$&$s^5$&$s^6$&$s^7$&$s^8$&$s^9$&$s^{10}$&$s^{11}$&$s^{12}$&$s^{13}$&$s^{14}$&$s^{15}$&$s^{16}$&$s^{17}$&$s^{18}$\\
\hline
$\BS{e}_1$&$p_{10}$&$h_2$&$p_{01}$&$h_4$&$p_{20}$&$h_6$&$p_{11}$&$h_8$&$p_{30}$&$h_{10}$&$p_{21}$&$h_{12}$&$p_{40}$&$h_{14}$&$p_{03}$\\
$\BS{e}_2$&$$&&$q_{01}$&&$q_{20}$&$h_1$&$q_{11}$&$h_3$&$q_{30}$&$h_5$&$q_{21}$&$h_7$&$q_{40}$&$h_9$&$q_{03}$\\
\hline
\end{tabular}
\end{center}
\begin{center}
\begin{tabular}{|c|c|c|c|c|c|c|c|c|c|c|c|c|c|c|c|}
\hline
$W^\#_{1,5}$&$s^4$&$s^5$&$s^6$&$s^7$&$s^8$&$s^9$&$s^{10}$&$s^{11}$&$s^{12}$&$s^{13}$&$s^{14}$&$s^{15}$&$s^{16}$&$s^{17}$&$s^{18}$\\
\hline
$\BS{e}_1$&$p_{10}$&$h_2$&$p_{01}$&$h_4$&$p_{20}$&$h_6$&$p_{11}$&$h_8$&$p_{30}$&$h_{10}$&$p_{21}$&$h_{12}$&$p_{40}$&$h_{14}$&$p_{03}$\\
$\BS{e}_2$&$$&&$q_{01}$&&$q_{20}$&&$q_{11}$&$h_1$&$q_{30}$&$h_3$&$q_{21}$&$h_5$&$q_{40}$&$h_7$&$q_{03}$\\
\hline
\end{tabular}
\end{center}
Alternative processes to determine $p_{i,j}$, $q_{i,j}$ and $h_k$ are as follows:
\begin{center}
\begin{tabular}{|c|c|c|c|c|c|c|c|c|c|c|c|c|c|c|c|}
\hline
$W^\#_{1,1}$
&$s^4$&$s^5$&$s^6$&$s^7$&$s^8$&$s^9$&$s^{10}$&$s^{11}$&$s^{12}$&$s^{13}$&$s^{14}$&$s^{15}$&$s^{16}$&$s^{17}$&$s^{18}$\\
\hline
$\BS{e}_1$&$p_{10}$&$h_2$&$h_3$&$p_{01}$&$p_{20}$&$h_6$&$h_7$&$p_{11}$&$p_{30}$&$h_{10}$&$h_{11}$&$p_{21}$&$p_{40}$&$h_{14}$&$p_{03}$\\
$\BS{e}_2$&&&$q_{01}$&$h_1$&$q_{20}$&$h_4$&$h_5$&$q_{11}$&$q_{30}$&$h_8$&$h_9$&$q_{21}$&$q_{40}$&$h_{12}$&$q_{03}$\\
\hline
\end{tabular}
\end{center}
\begin{center}
\begin{tabular}{|c|c|c|c|c|c|c|c|c|c|c|c|c|c|c|c|}
\hline
$W^\#_{1,3}$&$s^4$&$s^5$&$s^6$&$s^7$&$s^8$&$s^9$&$s^{10}$&$s^{11}$&$s^{12}$&$s^{13}$&$s^{14}$&$s^{15}$&$s^{16}$&$s^{17}$&$s^{18}$\\
\hline
$\BS{e}_1$&$p_{10}$&$h_2$&$h_3$&$h_4$&$p_{20}$&$p_{01}$&$h_7$&$h_8$&$p_{30}$&$p_{11}$&$h_{11}$&$p_{21}$&$p_{40}$&$h_{14}$&$p_{03}$\\
$\BS{e}_2$&$$&&$q_{01}$&&$q_{20}$&$h_1$&$h_5$&$h_6$&$q_{30}$&$q_{11}$&$h_9$&$q_{21}$&$q_{40}$&$h_{12}$&$q_{03}$\\
\hline
\end{tabular}
\end{center}
\begin{center}
\begin{tabular}{|c|c|c|c|c|c|c|c|c|c|c|c|c|c|c|c|}
\hline
$W^\#_{1,5}$&$s^4$&$s^5$&$s^6$&$s^7$&$s^8$&$s^9$&$s^{10}$&$s^{11}$&$s^{12}$&$s^{13}$&$s^{14}$&$s^{15}$&$s^{16}$&$s^{17}$&$s^{18}$\\
\hline
$\BS{e}_1$&$p_{10}$&$h_2$&$h_3$&$h_4$&$p_{20}$&$h_6$&$h_7$&$p_{01}$&$p_{30}$&$h_{10}$&$h_{11}$&$p_{11}$&$p_{40}$&$h_{14}$&$p_{03}$\\
$\BS{e}_2$&$$&&$q_{01}$&&$q_{20}$&&$h_5$&$h_1$&$q_{30}$&$h_8$&$h_9$&$q_{11}$&$q_{40}$&$h_{12}$&$q_{03}$\\
\hline
\end{tabular}
\end{center}
(iii) We assume that $\theta_1=\theta_2=0$ and $\theta_3\ne0$.
Then, by \eqref{Taylor:phi}, we have 
\begin{equation*}
\phi_4(s)=\BS{e}_1\frac{s^4}{4!}+\theta_3\BS{e}_2\frac{\binom{6}3s^7}{7!}+o(s^7).
\end{equation*}
Its coefficient of $\BS{e}_2$ is given by 
$$
\sum_{i=4}^{12}\theta_{i-4}\frac{\binom{i-1}3s^i}{i!}+(\theta_9-280\theta_3^3)\frac{\binom{12}3s^{13}}{13!}+o(s^{13}).
$$
Setting $S=\{4i+7j: i,j\in\BB{Z}_{\ge},i+j\ge1\}$, we have 
$$
S=\{4, 7, 8, 11, 12, 14, 15, 16\}\cup\{i\in\BB{Z}: i\ge18\}. 
$$
We set $h_2=h_3=0$, in order not to change the coefficients of $\BS{e}_1s^5$ and $\BS{e}_1s^6$. 

By Lemma \ref{Lemma:Key}, we have the following:
\begin{itemize}
\item
For $e\in S$, we can choose $p_{i,j}$  to send the coefficients of $\BS{e}_1s^e$ to  $0$ (resp.~$1$) if $e\ne4$ (resp.~$e=4$) (by Lemma \ref{Lemma:Key}, (i)).
\item
For $e\in S$, we can choose $q_{i,j}$  to send the coefficients of $\BS{e}_2s^e$ to  $0$ (resp.~$1$) if $e\ne7$ (resp.~$e=7$) (by Lemma \ref{Lemma:Key}, (i))
\item
We can choose $h_6$, to send the coefficients of $\BS{e}_1s^9$ to zero (by Lemma \ref{Lemma:Key}, (iii)).
\item
The coefficient of  $\BS{e}_2s^9$ is constant multiple of $\theta_5$ (by  Lemma \ref{Lemma:Key}, (iii)).
\item
We can choose $h_{10}$, to send the coefficients of $\BS{e}_1s^{13}$ to zero (by Lemma \ref{Lemma:Key}, (iii)).
\item
We can choose $h_{11}$ and $h_{14}$, to send the coefficients of $s^{17}$ to zero (by  Lemma \ref{Lemma:Key}, (iii)).
\item
When $\theta_5\ne0$, we can choose $h_1$ to send the coefficient of $\BS{e}_2s^9$ to $\pm1$. 
Then, we can send the coefficient of $\BS{e}_1s^{13}$ to $0$, choosing $h_5$ suitably. 
\item
We remark that the coefficient of  $\BS{e}_2s^{13}/13!$ becomes $11h_1^6w_{18}$, when we align the coefficients 
under the condition $\theta_5=0$. Then, we can choose $h_1$ to send the coefficient of $\BS{e}_2s^{13}$ to $\pm1$.
\end{itemize}

Applying these processes repeatedly, we complete the proof. 
The processes to determine $p_{i,j}$, $q_{i,j}$ and $h_k$ are summarized as the following table:
\begin{center}
\begin{tabular}{|c|l|c|c|c|c|c|c|c|c|c|c|c|c|c|c|}
\hline
$W_{18}$&$s^4$&$s^5$&$s^6$&$s^7$&$s^8$&$s^9$&$s^{10}$&$s^{11}$&$s^{12}$&$s^{13}$&$s^{14}$&$s^{15}$&$s^{16}$&$s^{17}$&$s^{18}$\\
\hline
$\BS{e}_1$&$p_{10}$&$h_2$&$h_3$&$p_{01}$&$p_{20}$&$h_6$&$h_7$&$p_{11}$&$p_{30}$&$h_{10}$&$p_{02}$&$p_{21}$&$p_{40}$&$h_{14}$&$p_{12}$\\
$\BS{e}_2$&&&&$q_{01}$&$q_{20}$&&$h_4$&$q_{11}$&$q_{30}$&&$q_{02}$&$q_{21}$&$q_{40}$&$h_{11}$&$q_{12}$\\
\hline
\end{tabular}
\end{center}
\end{proof}
\begin{rem}\label{remW12}
For $W_{12}$ singularity, setting $\Phi(x,y)=(\frac{4!}{h_1^4}x+\frac{10\theta_2}{h_1^4\theta_1^2}y,\frac{30}{h_1^5\theta_1}y)$, 
we observe that 
$$
\Phi\comp\phi_4\comp h(t)=(t^4,t^5+\tfrac{5h_1^2w_{12}}{42}t^7+o(t^7))
$$
where $h(t)=h_1t
- ( \frac{{{h}_{1}^{2}} {{\theta }_2}}{6 {{\theta }_1}}) \frac{t^2}{2!}
- ( \frac{{{h}_{1}^{3}}\, \left( {{\theta }_{2}^{2}}\mathop{-}8 {{\theta }_{1}^{4}}\right) }{16 {{\theta }_{1}^{2}}}) \frac{t^3}{3!}
- ( \frac{{{h}_{1}^{4}} {{\theta }_2} ( 5 {{\theta }_1} {{\theta }_3}\mathop{-}7 {{\theta }_{2}^{2}}\mathop{-}13 {{\theta }_{1}^{4}}) }{21 {{\theta }_{1}^{3}}}) \frac{t^4}{4!}
$. 
This explains why we need $w_{12}$ for $W_{12}$ singularity. 
A similar argument applies to $W_{18}$ singularity, but we omit the details here.
\end{rem}
\section{Parallel curves}\label{5}
Since the nineteenth century, parallel curves have attracted considerable interest 
(see \cite{Cayley}, \cite{Roberts}, for example), 
and it has long been recognized that singularities arise at specific points under parallel translation.
To the best of the authors' knowledge, Bruce and Giblin \cite[§7.12]{BGib} were the first to identify, within a modern singularity-theoretic framework, the appearance of an $A_2$ singularity (a 3/2-cusp) on a parallel curve at a non-vertex point.
Porteous \cite[Proposition 1.16]{Porteous} later described a criterion for detecting an $E_6$ singularity (a 4/3-cusp), observing that such cusps occur on parallel curves associated with first-order vertices, in connection with the evolute.
Our approach extends this analysis to degenerate vertices (Theorem~\ref{Thm:Parallel1}, (iii)), and also to situations in which the original curve itself may possess singularities (Theorems~\ref{Thm:Parallel2} and~\ref{Thm:Parallel3}).
We determine the conditions under which the parallel curve of a plane curve with $\mathcal A$-simple singularities remains $\mathcal A$-simple.

We consider the parallel curves of $\phi_m(t)$, \eqref{Def:Phi_m}, 
defined by $\phi_m^\delta(t)=\phi_m(t)+\delta\BS{n}$ where  $\delta$ is a non-zero constant.
We assume that $t$ is a curvature parameter of $\phi_m$. 
\begin{lem}\label{Lem:MultParallel}
The multiplicity $m^\delta$ of the parallel $\phi_m^\delta(t)$ at $t=0$ is given by the following:
$$
m^\delta=
\begin{cases}
\ord \theta,&\text{if}\ \ord \theta<m;\\
m,&\text{if}\ \ord \theta>m,\ \text{or}\ \ord \theta=m,\ \delta\ne\theta_m^{-1};\\
\min\{i:\theta_i\ne0, i>\ord \theta\},&\text{if}\ \ord \theta=m,\ \delta=\theta_m^{-1}.
\end{cases}
$$
\end{lem}
\begin{proof}
This is a consequence of the following computation:
\begin{align*}
(\phi_m^\delta)'(t)
=&\Bigl(\frac{t^{m-1}}{(m-1)!}-\delta\kappa\Bigr)\BS{t}\\
=&\Bigl(\frac{t^{m-1}}{(m-1)!}-\delta\sum_{i\ge\ord \theta}\theta_i\frac{t^{i-1}}{(i-1)!}\Bigr)\BS{t}\\
=&\Bigl((1-\delta\theta_m)\frac{t^{m-1}}{(m-1)!}-\delta\sum_{i\ge\ord \theta,\ i\ne m}\theta_i\frac{t^{i-1}}{(i-1)!}\Bigr)\BS{t}.
\qedhere\end{align*}
\end{proof}
Lemma \ref{Lem:MultParallel} has the following consequences: 
\begin{rem}\label{rmk:multofP}
By a suitable choice of a nonsingular plane curve, its parallel curve at the radius of curvature can degenerate into a singular curve of arbitrarily high multiplicity.
On the other hand, the multiplicities of parallel curves of $A_{2k}$ (resp.~$E_{6k}$, $E_{6k+2}$) 
singularities are at most $2k-1$ (resp.~$3k-2$, $3k-1$). 
Moreover, the multiplicities of parallel curves of $W_{12}$ (resp.~$W^{\#}_{1,2q-1}$, $W_{18}$) 
singularities are always $1$ (resp.~$2$, $3$).
\end{rem}
\begin{rem}
Cayley~\cite{Cayley} and Roberts~\cite{Roberts} discussed the degrees of parallel curves of certain algebraic curves.
Lemma~\ref{Lem:MultParallel} may be regarded as a local version of their arguments.
Their works also suggest an intention to analyze the role of singularities, although this aspect does not seem to have been fully developed.
Since the present section analyzes singularities of parallel curves, we hope that this work may contribute, even in a small way, to research in that direction.
\end{rem}
Let $t$ denote the curvature parameter of $\phi_m$.
Since $\frac{d}{dt}(\phi^\delta_m)=\frac{d}{dt}\phi_m-\delta\kappa\BS{t}$, a solution  $s$  to 
$$
\pm\frac{s^{m^\delta-1}}{(m^\delta-1)!}\frac{ds}{dt}=\frac{t^{m-1}}{(m-1)!}-\delta\kappa,
$$
is a curvature parameter of the parallel curve $\phi_m^\delta$.  
We choose the sign on the left-hand side so that the orientation induced by $s$ agrees with that induced by 
$t$.
We thus conclude that 
\begin{equation}\label{Form:CurvParaPC}
\pm\frac{s^{m^\delta}}{m^\delta!}
=\frac{t^m}{m!}-\delta\theta=(1-\delta\theta_m)\frac{t^m}{m!}-\delta\sum_{i\ne m}\theta_i\,\frac{t^i}{i!}.
\end{equation}
Substituting $t$ by a power series in the curvature parameter $s$, say $t_0s(1+\sum_{i\ge1}t_i\frac{s^i}{(i+1)!})$,  
into the right hand side of \eqref{Form:CurvParaPC}, we can determine $t_i$, $i=0, 1, 2, \dots$, sequentially.

Note that the curve $\phi_m$ and its parallel curve $\phi_m^\delta$ share the same frame $\{\BS{t},\BS{n}\}$, 
and hence the same angle function $\theta$. Therefore, we have 
$$
\theta=\sum_{i=1}^r\theta_i\frac{t^i}{i!}+o(t^r)=\sum_{i=1}^r\theta^\delta_{i}\frac{s^i}{i!}+o(s^r)
$$
where $t$ is the curvature parameter of $\phi_m$ and $s$ is the curvature parameter of $\phi_m^\delta$. 
Here $\theta^\delta_i$ denotes the corresponding invariants of the parallel curve $\phi_m^\delta$. 
Throughout, the superscript $\delta$ is used to indicate invariants associated with the parallel curve $\phi_m^\delta$.
For the parallel curve of a nonsingular curve, it is classically well known that 
\[
\kappa^\delta=\frac{\kappa}{1-\delta\kappa}. 
\]

For $m=1,2,3,4$, we express several $\theta^\delta_i$'s in terms of $\theta_i$, $i=1,2,\dots$, and $\delta$, 
in the proofs of the subsequent theorems. 
These are generalizations of the classical curvature relation above and 
enable us to determine when the parallel curve has an $\mathcal A$-simple singularity.
\begin{thm}\label{Thm:Parallel1}
The singularities of parallel curves of a nonsingular curve ($m=1$) are described as follows: 
\begin{itemize}
\item[\rm (i)] 
If $\phi_1$ is neither an inflection point nor a vertex at $0$ (i.e., $\theta_1\theta_2\ne0$), then $\phi_1^\delta$, $\delta=\theta_1^{-1}$, is an $A_2$ singularity at $0$. 
\item[\rm (ii)] 
If $\phi_1$ is a non inflection 1st order vertex (i.e., $\theta_1\ne0$, $\theta_2=0$, $\theta_3\ne0$), then 
$\phi_1^\delta$, $\delta=\theta_1^{-1}$, is an $E_6$ singularity at $0$.
\item[\rm (iii)] 
If $\phi_1$ is a non inflection 2nd order vertex (i.e., $\theta_1\ne0$, $\theta_2=\theta_3=0$, $\theta_4\ne0$), then 
$\phi_1^\delta$, $\delta=\theta_1^{-1}$, is a $W_{12}$ singularity at $0$.
Moreover, we have
\begin{itemize}
\item $\phi_1^\delta$ $\sim_{\mathcal A}$ $(t^4,t^5\pm t^7)$ if $\pm\tilde{w}_{12}^\delta>0$, 
\item $\phi_1^\delta$ $\sim_{\mathcal A}$ $(t^4,t^5)$ if $\tilde{w}_{12}^\delta=0$, 
\end{itemize}
where $\tilde{w}_{12}^\delta=\theta_4\theta_6-\frac{35}{48}\,\theta_5^2+50\,\theta_1^2\theta_4^2$.
\end{itemize}
\end{thm}
\begin{proof}
First we remark that, if $\delta\theta_1\ne1$, then $m^\delta=1$, 
$t=t_0s+\frac{\delta\theta_2}{1-\delta\theta_1}{(t_0s)^2}/2+o(s^2)$, $t_0=|1-\delta\theta_1|^{-1}$,
$$
\theta^\delta_1=\frac{\theta_1}{1-\delta\theta_1},\ 
\theta^\delta_2=\frac{t_0\theta_2}{(1-\delta\theta_1)^2},\
\theta^\delta_3=\frac{\theta_3}{(1-\delta\theta_1)^2}+\frac{3t_0^2\delta\theta_2^2}{(1-\delta\theta_1)^3}. 
$$

(i): When $\delta\theta_1=1$ and $\theta_2\ne0$, we have $m^\delta=2$,  and 
$t=t_0s-\frac{\theta_3}{3\theta_2}{(t_0s)^2}/2+o(s^2)$, 
$t_0=|\theta_1/\theta_2|^{\frac12}$, 
$$
\theta^\delta_1=-\frac{\theta_1^2}{t_0\theta_2},\ 
\theta^\delta_2=\theta_1\Bigl(\frac{\theta_1\theta_3}{3\theta_2^2}-1\Bigr), \
\theta^\delta_3=\frac{t_0\theta_1^2}{12\theta_2^3}(3\theta_2\theta_4-5\theta_3^2).
$$
Since $\theta_1\ne0$, we have $\theta_1^\delta\ne0$ and $\phi_1^\delta$ is $A_2$ singularity.

(ii): When $\delta\theta_1=1$, $\theta_2=0$ and $\theta_3\ne0$, we have $m^\delta=3$, and 
$t=t_0s-\frac{\theta_4}{6\theta_3}{(t_0s)^2}/2+o(s^2)$, $t_0=|{\theta_1}/{\theta_3}|^{\frac13}$, 
$$
\theta^\delta_1=-\frac{\theta_1^2}{t_0^2\theta_3},\ 
\theta^\delta_2=\frac{\theta_1^2\theta_4}{6t_0\theta_3^2},\ 
\theta^\delta_3=-\theta_1\Bigl(\frac{\theta_1\theta_5}{10\theta_3^2}-\frac{\theta_1\theta_4^2}{8\theta_3^3}-1\Bigr).
$$
Since $\theta_1\ne0$, we have $\theta_1^\delta\ne0$ and $\phi_1^\delta$ is $E_6$ singularity.

(iii): When $\delta\theta_1=1$, $\theta_2=\theta_3=0$ and $\theta_4\ne0$, we have $m^\delta=4$,  and 
$t=t_0s-\frac{\theta_5}{10\theta_4}{(t_0s)^2}/2+\cdots$, 
$t_0=|{\theta_1}/{\theta_4}|^{\frac14}$.
We thus conclude that 
$$
\theta^\delta_1=-\frac{\theta_1}{t_0^3\theta_4},\ 
\theta^\delta_2=\frac{\theta_1\theta_5}{10t_0^2\theta_4^2},\ 
\theta^\delta_3=\frac{\theta_1^2}{20t_0\theta_4^3}\Bigl(\theta_4\theta_6-\frac{21}{20}\theta_5^2\Bigr).
$$
We claim the last assertion, since we can extract that 
\[
w_{12}^\delta=\theta^\delta_1\theta^\delta_3-\tfrac{77}{48}\,(\theta_2^\delta)^2+\tfrac52\,(\theta^\delta_1)^4
=\frac{|{\theta_1}/{\theta_4}|^3\,\tilde{w}^\delta_{12}}{20}.
\qedhere
\]
\end{proof}
\begin{thm}\label{Thm:Parallel2}
The singularities of parallel curves of $\mathcal A$-simple singularities with multiplicity 2 ($m=2$) are described as follows:
\begin{itemize}
\item[\rm (i)] If $\phi_2$ is an $A_2$ singularity  at $0$ (i.e., $\theta_1\ne0$), then $\phi_2^\delta$, $\delta\ne0$,  is nonsingular at $0$.
\item[\rm (ii)] If $\phi_2$ is an $A_4$ singularity  at $0$ (i.e., $\theta_1=0$, $\theta_3\ne0$), then $\phi_2^\delta$
is $A_4$ (resp.~$E_8$) singularity  at $0$ if $\delta\ne\theta_2^{-1}$ (resp.~$=\theta_2^{-1}$).
\item[\rm (iii)] 
Assume that $\phi_2$ is an $A_{2k}$ singularity  at $0$ with $k\ge3$.
\begin{itemize}
\item[\rm (iii-a)] If $\delta\ne\theta_2^{-1}$, then $\phi_2^\delta$ is $A_{2k}$ singularity at $0$.
\item[\rm (iii-b)] If $\delta=\theta_2^{-1}$ and  $\theta_4\ne0$, then  $\phi_2^\delta$ is $W_{1,2k-5}^\#$ singularity at $0$.
\item[\rm (iii-c)] If $\delta=\theta_2^{-1}$ and  $\theta_4=0$, then  $\phi_2^\delta$ is 
a singularity of multiplicity $\ge$5 at $0$. In particular, $\phi_2^\delta$ is not $\mathcal A$-simple at $0$. 
\end{itemize}
\end{itemize}
\end{thm}
\begin{proof}
(i): This case has already been trated in Theorem \ref{Thm:Parallel1} (i). 
But we present some computation to see the invariants $\theta_i^\delta$.
When $\theta_1\ne0$, we have $m^\delta=1$, $t=t_0s+\frac{1-\delta\theta_2}{\delta\theta_1}(t_0s)^2/2+o(s^2)$, $t_0=|1-\delta\theta_1|^{-1}$, 
$$
\theta^\delta_1=-\frac1{\delta},\ 
\theta^\delta_2=-\frac{t_0}{\delta^2\theta_1},\
\theta^\delta_3=-\frac{t_0^2}{\delta^3\theta_1^2}(-2+3\delta\theta_2-\frac1{1-\delta\theta_1}).
$$

(ii): When $\theta_1=0$ and $\delta\theta_2\ne1$, we have $m^\delta=2$, 
$t=t_0s-\frac{\delta\theta_3}{3(1-\delta\theta_2)^2}{(t_0s)^2}/2+o(s^2)$, 
$t_0=|1-\delta\theta_2|^{-\frac12}$, 
$$
\theta^\delta_1=0,\ 
\theta^\delta_2=\frac{\theta_2}{1-\delta\theta_2},\ 
\theta^\delta_3=\frac{t_0\theta_3}{(1-\delta\theta_2)^2},\ 
$$
and, if $\theta_3=0$, then 
$\theta^\delta_5=t_0^3\theta_5(1-\delta\theta_2)^{-2}$.

When $\theta_1=0$,  $\delta\theta_2=1$ and $\theta_3\ne0$, we have $m^\delta=3$, 
$t=t_0s-\frac{\theta_4}{\theta_3}{(t_0s)^2}/2+o(s^2)$, $t_0=|{\theta_2}/{\theta_3}|^{\frac13}$,
$$
\theta^\delta_1=0,\ 
\theta^\delta_2=-\frac{\theta_2^2}{t_0\theta_3},\
\theta^\delta_3=\theta_2\Bigl(\frac{\theta_2\theta_4}{2\theta_3^2}-1\Bigr),\ 
$$
Since $\theta_2\ne0$, we have $\theta_2^\delta\ne0$ and $\phi_2^\delta$ is $E_8$ singularity.

(iii):  Assume that $\delta\ne\theta_2^{-1}$. 
Then, by \eqref{Form:CurvParaPC}, we have 
$$
\pm\frac{s^2}2=(1-\delta\theta_2)\frac{t^2}2-\delta\sum_{i\ge4}\theta_i\frac{t^i}{i!}.
$$
Setting $t=\sigma\tau s$, $\sigma=|1-\delta\theta_2|^{-\frac12}$, we have 
$$
\pm\frac{s^2}2=\frac{\tau^2s^2}2
-\delta\theta_4\frac{(\sigma\tau s)^4}{4!}
-\cdots
-\theta_{2k-2}\frac{(\sigma\tau s)^{2k-2}}{(2k-2)!}
-\theta_{2k-1}\frac{(\sigma\tau s)^{2k-1}}{(2k-1)!}
-\cdots.
$$
Then we can set $\tau=1+\sum_{i\ge1}t_{i+1}s^i/i!$, and we conclude that the first odd oder term of $\tau$ is $\theta_{2k-1}\frac{\sigma^{2k-1}}{(2k-1)!}s^{2k-1}$.
Then we obtain that
\begin{align*}
\theta=
&\theta_2\frac{(\sigma\tau s)^2}2
+\theta_4\frac{(\sigma\tau s)^4}{4!}
+\cdots
+\theta_{2k-2}\frac{(\sigma\tau s)^{2k-2}}{(2k-2)!}
+\theta_{2k-1}\frac{(\sigma\tau s)^{2k-1}}{(2k-1)!}
+\cdots
\\
=&\sigma^2\theta_2\frac{s^2}2+{\theta^\delta_4}\frac{s^4}{4!}+\cdots+\theta_{2k-2}^\delta\frac{s^{2k-2}}{(2k-2)!}
+\sigma^{2k-1}\theta_{2k-1}\frac{s^{2k-1}}{(2k-1)!}+\cdots
\end{align*}
which concludes (iii-a).

Assume that $\delta=\theta_2^{-1}$. 
It is enough to show that the condition \eqref{Condition:Ak}  imply
$$
\theta_1^\delta=\theta_3^\delta=\cdots=\theta_{2k-5}^\delta=0, \theta_{2k-3}^\delta\ne0
$$
whenever $\theta_4\ne0$. 
Let $t$ be a curvature parameter of $\phi_2$ and $s$ be that of $\phi_2^\delta$.
Then we have 
$$
\pm \frac{s^4}{4!}=\frac{t^2}2-\delta\theta=\CR{-}\theta_2^{-1}\sum_{i\ge4}\theta_i\frac{t^i}{i!}.
$$
Since $\delta=\theta_2^{-1}$, this implies that 
$$
\pm\frac{s^4}{4!}=
-\theta_2^{-1}
\Bigl(\theta_4\frac{t^4}{4!}+\theta_6\frac{t^6}{6!}+\cdots+\theta_{2k-2}\frac{t^{2k-2}}{(2k-2)!}+
\theta_{2k-1}\frac{t^{2k-1}}{(2k-1)!}+\cdots\Bigr). 
$$
Setting $t=\sigma\tau s$, $\sigma=|\theta_2/\theta_4|^{1/4}$, we have 
\begin{align*}
\mp\frac{1}{4!}
=&
\frac{\tau^4}{4!}
+\frac{\theta_6}{\theta_4}\frac{\tau^6(\sigma s)^2}{6!}
+\cdots
+\frac{\theta_{2k-2}}{\theta_4}\frac{\tau^{2k-2}(\sigma s)^{2k-6}}{(2k-2)!}
+\frac{\theta_{2k-1}}{\theta_4}\frac{\tau^{2k-1}(\sigma s)^{2k-5}}{(2k-1)!}+\cdots.
\end{align*}
We write this relation as 
\begin{equation}\label{Compare1}
\mp\frac{\tau^{-4}}{4!}=\frac1{4!}
+\frac{\theta_6}{\theta_4}\frac{\tau^2(\sigma s)^2}{6!}
+\cdots
+\frac{\theta_{2k-2}}{\theta_4}\frac{\tau^{2k-4}(\sigma s)^{2k-6}}{(2k-2)!}
+\frac{\theta_{2k-1}}{\theta_4}\frac{\tau^{2k-5}(\sigma s)^{2k-5}}{(2k-1)!}+\cdots.
\end{equation}
We conclude that we can write  
$$
\tau=1+a_1s^2+\cdots+a_{k-3}s^{2k-6}+bs^{2k-5}+\cdots.
$$
Comparing the coefficients of $s^{2k-5}$ in the both sides of the equation \eqref{Compare1}, we obtain 
$$
b=\pm\frac{\theta_{2k-1}}{\theta_4}\frac{3!\,\sigma^{2k-5}}{(2k-1)!}.
$$
Setting $\tau=1+s^2A(s^2)+s^{2k-5}B(s^2)$, 
\begin{align*}
\theta
=&\sum_{i\ge2}\theta_i\frac{t^i}{i!}
=\sum_{i\ge2}\theta_i\frac{\sigma^i\tau^is^i}{i!}\\
=&\sum_{i\ge1}\theta_{2i}\sigma^{2i}s^{2i}\sum_{i_0+i_1\le 2i}\frac{s^{2i_0+(2k-5)i_1}A(s^2)^{i_0}B(s^2)^{i_1}}{(2i-i_0-i_1)!i_0!i_1!}\\
+&\sum_{i\ge k-3}\theta_{2i+1}\sigma^{2i+1}s^{2i+1}\sum_{i_0+i_1\le 2i+1}\frac{s^{2i_0+(2k-5)i_1}A(s^2)^{i_0}B(s^2)^{i_1}}{(2i+1-i_0-i_1)!i_0!i_1!}.
\end{align*}
We remark that the coefficients of $s^i$ in this series yield $\theta^\delta_i$,  and we obtain the first nonzero 
$\theta^\delta_i$ with odd $i$ is 
\begin{align*}
\theta^\delta_{2k-3}
=\pm(2k-3)!\theta_2\sigma^2b
=\pm\theta_2\frac{\theta_{2k-1}}{\theta_4}\frac{3!\,(2k-3)!\,\sigma^{2k-3}}{(2k-1)!}.
\end{align*}
We thus conclude (iii-b). 

The item (iii-c) is a consequence of Lemma \ref{Lem:MultParallel}.
\end{proof}
\begin{figure}
\begin{center}
\includegraphics[width=0.5\textwidth,bb=0 0 360 321]{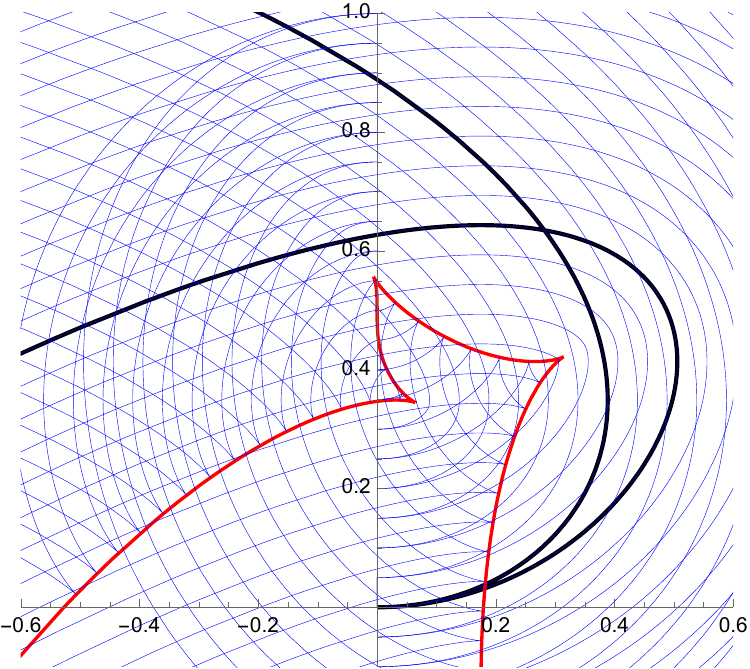}\\
\caption{Curve $(s^2/2 - s^6/12 + s^7/42 -7s^8/576,  s^4/4 - s^5/30 + s^6/72 - s^8/48)$ (thick), 
its parallel curves (blue) and the evolute (red).} 
\end{center}
\end{figure}
\begin{thm}\label{Thm:Parallel3}
The singularities of parallel curves of $\mathcal A$-simple singularities with multiplicity 3 ($m=3$) are described as follows:
\begin{itemize}
\item[\rm (i)] 
If $\phi_3$ is an $E_6$ singularity at $0$ (i.e., $\theta_1\ne0$), then 
$\phi_3^\delta$, $\delta\ne0$, is nonsingular at $0$.
\item[\rm (ii)] 
If $\phi_3$ is an $E_8$ singularity  at $0$ (i.e., $\theta_1=0$, $\theta_2\ne0$), then 
$\phi_3^\delta$, $\delta\ne0$, is an $A_4$ singularity at $0$. 
\item[\rm (iii)] 
Assume that $\phi_3$ is an $E_{12}$ singularity  at $0$ (i.e., $\theta_1=\theta_2=0$, $\theta_4\ne0$). 
\begin{itemize}
\item[\rm (iii-a)] If $\delta\ne\theta_3^{-1}$, then $\phi_3^\delta$ is a $E_{12}$ singularity at $0$. 
Moreover, we conclude that
$\phi_3^\delta$ is $\mathcal A$-equivalent to $(t^3,t^7+t^8)$ (resp.~$(t^3,t^7)$) at $0$ if 
$\delta(3\theta_3\theta_5-5\theta_4^2)-3\theta_5\ne0$ (resp.~$=0$).
\item[\rm (iii-b)] If $\delta=\theta_3^{-1}$, then $\phi_3^\delta$ is a $W_{18}$ singularity at $0$. 
Moreover, we conclude that $\phi_3^\delta$ is $\mathcal A$-equivalent to  
\begin{itemize}
\item $(t^4,t^7\pm t^9)$ at $0$ if $\mp(\theta_4\theta_6-\frac{27}{20}\,\theta_5^2)>0$,
\item $(t^4,t^7\pm t^{13})$ at $0$ if $\theta_4\theta_6-\frac{27}{20}\,\theta_5^2=0$ and $\mp\tilde{w}_{18}^\delta>0$,
\item $(t^4,t^7)$  at $0$ if $\theta_4\theta_6-\frac{27}{20}\,\theta_5^2=0$ and $\tilde{w}_{18}^\delta=0$, 
\end{itemize}
\end{itemize}
where 
$\tilde{w}_{18}^\delta=
\theta_4^5\theta_{10}
-\frac{13}2\theta_4^4\theta_5\theta_9
+\frac{117}{10}\theta_4^3\theta_5^2\theta_8
-\frac{13923}{2000}\theta_4^2\theta_5^3\theta_7
-\frac{273}{100}\theta_4^4\theta_7^2
-\frac{2269449}{160000}\theta_5^6
-\frac{8120}3\theta_3^2\theta_4^6$.
\item[\rm (iv)] 
Assume that $\phi_3$ is an $E_{14}$ singularity  at $0$ (i.e., $\theta_1=\theta_2=\theta_4=0$, $\theta_5\ne0$). 
Then $\phi_3^\delta$ is $\mathcal A$-equivalent to 
\begin{itemize}
\item[\rm (iv-a)] an $E_{14}$ singularity at $0$, more precisely, $\mathcal A$-equivalent to $(t^3,t^8+t^9)$ 
(resp.~$(t^3,t^8)$), 
if $\delta\ne\theta_3^{-1}$ and $\delta(2\theta_3\theta_7-7\theta_5^2)-2\theta_7\ne0$ (resp.~$=0$);
\item[\rm (iv-b)] a singularity of multiplicity 5  at $0$ if $\delta=\theta_3^{-1}$. 
\end{itemize}
\end{itemize}
\end{thm}
\begin{proof}
The cases (i), (ii) have been trated in Theorem \ref{Thm:Parallel1} (ii), \ref{Thm:Parallel2} (ii). 
But we present some computation to see the invariants $\theta_i^\delta$. 

(i): When $\theta_1\ne0$, we have $m^\delta=1$, 
$t=t_0s-\frac{\theta_2}{\theta_1}(t_0s)^2/2+o(s^2)$, $t_0=|\delta\theta_1|^{-1}$, 
$$
\theta^\delta_1=-\frac1{\delta},\ 
\theta^\delta_2=0,\
\theta^\delta_3=-\frac{t_0^2}{\delta^2\theta_1}.
$$
In this case, $\phi_3^\delta$ is nonsingular.

(ii):
When $\theta_1=0$ and $\theta_2\ne0$, we have $m^\delta=2$, 
$t=t_0s+\frac{1-\delta\theta_3}{3\delta\theta_2}(t_0s)^2/2+o(s^2)$, $t_0=|\delta\theta_2|^{-\frac12}$, 
$$
\theta^\delta_1=0,\ 
\theta^\delta_2=-\frac1\delta,\ 
\theta^\delta_3=-\frac{t_0}{\delta^2\theta_2}.  
$$
In this case, $\phi_3^\delta$ has $A_4$ singularity at $s=0$.

(iii): When $\theta_1=\theta_2=0$ and $\delta\theta_3\ne1$, we have $m^\delta=3$, 
$t=t_0s+\frac{\delta\theta_4}{6(1-\delta\theta_3)}(t_0s)^2/2+o(s^2)$, $t_0=|1-\delta\theta_3|^{-\frac13}$, 
$$
\theta^\delta_1=0,\ 
\theta^\delta_2=0,\ 
\theta^\delta_3=\frac{\theta_3}{1-\delta\theta_3},\ 
\theta^\delta_4=\frac{t_0\theta_4}{(1-\delta\theta_3)^2},\ 
\theta^\delta_5= t_0^2\Bigl(\frac{\theta_5}{(1-\delta\theta_3)^2}+\frac{5\delta\theta_4^2}{3(1-\delta\theta_3)^3}\Bigr). 
$$

If $\theta_4\ne0$, then $\theta_4^\delta\ne0$ and $\phi_3^\delta$ is $E_{12}$ singularity, and we have (iii-a).

When $\theta_1=\theta_2=0$, $\delta\theta_3=1$ and $\theta_4\ne0$, we have $m^\delta=4$, 
$t=t_0s+\frac{\theta_5}{10\theta_4}(t_0s)^2/2+o(s^2)$, $t_0=|{\theta_3}/{\theta_4}|^{\frac14}$,
$$
\theta^\delta_1=0,\ 
\theta^\delta_2=0,\
\theta^\delta_3=-\frac{\theta_3^2}{t_0\theta_4},\ 
\theta^\delta_4=\theta_3\Bigl(\frac{3\theta_3\theta_5}{5\theta_4^2}-1\Bigr),\ 
\theta^\delta_5=\frac{t_0\theta_3^2}{2\theta_4^3}\Bigl(\theta_4\theta_6-\frac{27}{20}\theta_5^2\Bigr).
$$
When $\theta_4\theta_6-\frac{27}{20}\theta_5^2=0$, we obtain that 
\begin{align*}
\theta^\delta_6=&\frac{t_0^2\theta_3^2}{\theta_4^4}\Bigl(\frac37\theta_4^2\theta_7-\frac{39}{40}\theta_5^3\Bigr),\\
\theta^\delta_9=&\frac{3t_0^5\theta_3^2}{10\theta_4^2}\Bigl(
\theta_{10}
-\frac{13\theta_5\theta_9}{2\theta_4}
+\frac{117\theta_5^2\theta_8}{10\theta_4^2}
+\frac{1989\theta_5^3\theta_7}{7\theta_4^3}
-\frac{39\theta_7^2}{7\theta_4}
-\frac{1155609\theta_5^6}{40000\theta_4^5}
\Bigr).
\end{align*}
We thus extract that 
$w_{18}^\delta=\theta^\delta_3\theta^\delta_9-\frac{4691}{1000}(\theta^\delta_6)^2+812(\theta^\delta_3)^4
=-\frac{3t_0^4\theta_3^4\tilde{w}_{18}}{10\theta_4^8}$.

(iv): If $\theta_1=\theta_2=\theta_4=0$ and $\theta_5\ne0$, then we have 
$$
\theta^\delta_5=\frac{t_0^2\theta_5}{(1-\delta\theta_3)^2},\ 
\theta^\delta_6=\frac{t_0^3\theta_6}{(1-\delta\theta_3)^2},\ 
\theta^\delta_7=t_0^4\Bigl(
\frac{\theta_7}{(1-\delta\theta_3)^2}+\frac{7\delta\theta_5^2}{2(1-\delta\theta_3)^3}\Bigr), 
$$
and we obtain the first item in (iv).

The second item of (iv) is a consequence of Lemma \ref{Lem:MultParallel}.
\end{proof}
\begin{rem}\label{SexCurv}
If $\phi_3$ has an $E_{12}$ (resp.~$E_{14}$) singularity,  
which is $\mathcal A$-equivalent to $(t^3,t^7+t^8)$ (resp.~$(t^3,t^8+t^9))$, that is, $\theta_5\ne0$ 
(resp.~$\theta_7\ne0$), then 
the singularity of the parallel curve $\phi_3^\delta$ degenerates to a singularity $\mathcal A$-equivalent to 
$(t^3,t^7)$ (resp.~$(t^3,t^8)$) when $\delta^{-1}$ is equal to
\begin{equation}\label{SecondCurv} 
\theta_3-\frac53\frac{\theta_4^2}{\theta_5}\quad
\Bigl(\text{resp.~}\theta_3-\frac72\frac{\theta_5^2}{\theta_7}\Bigr).
\end{equation}
This quantity can be regarded as an analogue of curvature.
We refer to the phenomena in (iii-a) (resp.~(iv-a)) as an {\bf equi-multiple degeneration}, and call the quantity \eqref{SecondCurv} the {\bf pseudo-curvature} of $E_{12}$ (resp.~$E_{14}$) singularity. 
\end{rem}
We generarize this phenomenon to the parallel curves of $E_{6k}$ and $E_{6k+2}$ singularities for $k\ge3$ as follows: 
\begin{thm}\label{E6kandE6k+2}
Assume that $\phi_3$ has an $E_{6k}$ (resp.~$E_{6k+2}$) singularity at $0$ with $k\ge3$. 
\begin{itemize}
\item Assume that $\delta\ne\theta_3^{-1}$. Then $\phi_3^\delta$ is an $E_{6k}$ (resp.~$E_{6k+2}$) singularity at $0$. 
More precisely, if $\phi_3$ is $\mathcal A$-equivalent to 
\begin{equation}\label{NF:E6k}
(t^3,t^{3k+1}+\varepsilon_pt^{3(k+p)+2})\quad(\text{resp.}~(t^3,t^{3k+1}+\varepsilon_pt^{3(k+p)+2}))\ \text{at $0$},
\end{equation}
then $\phi_3^\delta$ is $\mathcal A$-equivalent to $\phi_3$ at $0$
whenever \eqref{Condition:E_6k}  (resp.~\eqref{Condition:E_6k+2}) holds for  $0\le p<k-2$.
When \eqref{Condition:E_6k}  (resp.~\eqref{Condition:E_6k+2}) holds for  $p=k-2$,
$\phi_3^\delta$ has an equi-multiple degeneration, that is, it degenerates to 
$(t^3,t^{3k+1})$ (resp.~$(t^3,t^{3k+2})$)  at $0$ if and only if $\delta^{-1}$ satisfies 
$$
\tfrac1{\delta^{-1}-\theta_3}\,
\tfrac{2\theta_{3l_1+1}^2}{(3l_1)!(3l_1+1)!}
+\tfrac{\theta_{3l_2+2}}{(3l_2+2)!}=0\quad 
(\text{resp.} \ 
\tfrac{\theta_{3l_1+1}}{(3l_1+1)!}+\tfrac1{\delta^{-1}-\theta_3}\,
\tfrac{2\theta_{3l_2+2}^2}{(3l_2+1)!(3l_2+2)!}=0),
$$
where
\begin{equation}\label{L1L2}
(l_1,l_2)
=(k-1,k+p-1) \quad (\text{resp.}\ (k+p,k-1)).
\end{equation}
\item If $\delta^{-1}=\theta_3$, $\phi_3^\delta$ is of multiplicity $\ge5$ at $0$. 
In particular, $\phi_3^\delta$ is not $\mathcal A$-simple. 
\end{itemize}
\end{thm}
\begin{proof}
These are consequences of the following lemma.
\end{proof}
\begin{lem}
In the setup of Theorem \ref{E6kandE6k+2}, we have the following:
\begin{align}
\theta^\delta_3=&t_0^3\theta_3\quad
\text{where   \ $t_0=(1-\delta\theta_3)^{-\frac13}$}.
\label{LL3}
\end{align}
For $E_{6k}$ singularity, we have 
\begin{align*}
\theta^\delta_{3i+1}=&
\begin{cases}
0,&0\le i<l_1,\\
\frac{\theta_{3l_1+1}t_0^{3l_1	+1}}{1-\delta\theta_3},&i=l_1,
\end{cases}
\\
\theta^\delta_{3i+2}
=&
\begin{cases}
0,&0\le i<l_2,\\
\frac{\theta_{3l_2+2}t_0^{3l_2+2}}{1-\delta\theta_3},&i=l_2, 0\le p<k-2,\\
\frac{t_0^{3l_2+2}}{1-\delta\theta_3}\Bigl(
\theta_{3l_2+2}
+\frac{2(3l_2+2)!\,\delta\,\theta_{3l_1+1}^2}{(3l_1)!(3l_1+1)!(1-\delta\theta_3)}
\Bigr),
&i=l_2, p=k-2.
\end{cases}
\end{align*}
For $E_{6k+2}$ singularity, we have 
\begin{align*}
\theta^\delta_{3i+1}=&
\begin{cases}
0,&0\le i<l_1,\\
\frac{\theta_{3l_1+1}t_0^{3l_1	+1}}{1-\delta\theta_3},&i=l_1,0\le p<k-2,\\
\frac{t_0^{3l_1+1}}{1-\delta\theta_3}\Bigl(
\theta_{3l_1+1}
+
\frac{2(3l_1+1)!\,\delta\theta_{3l_2+2}^2}{(3l_2+1)!(3l_2+2)!(1-\delta\theta_3)}
\Bigr),
&i=l_1, p=k-2,
\end{cases}
\\
\theta^\delta_{3i+2}
=&
\begin{cases}
0,&0\le i<l_2,\\
\frac{\theta_{3l_2+2}t_0^{3l_2+2}}{1-\delta\theta_3},&i=l_2.
\end{cases}
\end{align*}
\end{lem}
\begin{proof}
By \eqref{Condition:E_6k}, \eqref{ia} and \eqref{ib} 
(resp.~\eqref{Condition:E_6k+2}, \eqref{iia} and \eqref{iib}), we can set 
\begin{equation*}
\theta=\theta_3\frac{t^3}{3!}+t^6\hat\theta_0(t^3)+t^{3l_1+1}\hat\theta_1(t^3)+t^{3l_2+2}\hat\theta_2(t^3)
\end{equation*}
where 
$\hat\theta_0(t^3)=\sum_{i\ge2}\theta_{3i}\frac{(t^3)^{i-2}}{(3i)!}$,
$\hat\theta_1(t^3)=\sum_{i\ge l_1}\theta_{3i+1}\frac{(t^3)^{i-l_1}}{(3i+1)!}$,
$\hat\theta_2(t^3)=\sum_{i\ge l_2}\theta_{3i+2}\frac{(t^3)^{i-l_2}}{(3i+2)!}$, 
assuming \eqref{L1L2}.

Setting 
$$
t=t_0s[1+s^3T_0(s^3)+s^{3l'_1+1}T_1(s^3)+s^{3l'_2+2}T_2(s^3)],
$$
we have that 
\begin{align}
\frac{t^i}{i!}=&t_0^is^i
\sum_{i_0+i_1+i_2\le i}
s^{3i_0+(3l_1'+1)i_1+(3l_2'+2)i_2}
\frac{T_0(s^3)^{i_0}T_1(s^3)^{i_1}T_2(s^3)^{i_2}}{(i-i_0-i_1-i_2)!i_0!i_1!i_2!}. 
\label{TTii}
\end{align}
Since $\frac{s^3}{3!}=\frac{t^3}{3!}-\delta\theta$, we obtain that 
\begin{align}
\frac{s^3}{3!}
=&(1-\delta\theta_3)\frac{t^3}{3!}
-\delta\Bigl[\sum_{i\ge2}\theta_{3i}\frac{t^{3i}}{(3i)!}
+\sum_{i\ge l_1}\theta_{3i+1}\frac{t^{3i+1}}{(3i+1)!}
+\sum_{i\ge l_2}\theta_{3i+2}\frac{t^{3i+2}}{(3i+2)!}\Bigr]
\label{STT}\\
=&
(1-\delta\theta_3)t_0^3s^3
\begin{bmatrix}
\frac16+\frac12s^3T_0(s^3)+\frac12s^6T_0(s^3)^2+\cdots\\
+s^{3l'_1+1}T_1(s^3)(\frac12+s^3T_0(s^3)+\cdots)+\frac12s^{6l'_2+4}T_2(s^3)^2+\cdots\\
+s^{3l'_2+2}T_2(s^3)(\frac12+s^3T_0(s^3)+\cdots)+\frac12s^{6l'_1+2}T_1(s^3)^2+\cdots
\end{bmatrix}
\notag\\
-&
\delta\sum_{i\ge2}\theta_{3i}t_0^{3i}s^{3i}
\begin{bmatrix}
\frac1{(3i)!}+\frac1{(3i-1)!}s^3T_0(s^3)+\frac1{(3i-2)!2}s^6T_0(s^3)^2+\cdots\\
+s^{3l'_1+1}T_1(s^3)(\frac1{(3i-1)!}+\frac1{(3i-2)!}s^3T_0(s^3)+\cdots)+\cdots\\
+s^{3l'_2+2}T_2(s^3)(\frac1{(3i-1)!}+\frac1{(3i-2)!}s^3T_0(s^3)+\cdots)+\cdots
\end{bmatrix}
\notag\\
-&
\delta\sum_{i\ge l_1}\theta_{3i+1}t_0^{3i+1}s^{3i+1}
\begin{bmatrix}
s^{3l'_2+2}T_2(s^3)(\frac1{(3i)!}+\frac1{(3i-1)!}s^3T_0(s^3)+\cdots)+\cdots\\
+\frac1{(3i+1)!}+\frac1{(3i)!}s^3T_0(s^3)+\frac1{(3i-1)!2}s^6T_0(s^3)^2+\cdots\\
+s^{3l'_1+1}T_1(s^3)(\frac1{(3i)!}+\frac1{(3i-2)!}s^3T_0(s^3)+\cdots)+\cdots\\
\end{bmatrix}
\notag\\
-&
\delta\sum_{i\ge l_2}\theta_{3i+2}t_0^{3i+2}s^{3i+2}
\begin{bmatrix}
s^{3l'_1+1}T_1(s^3)(\frac1{(3i+1)!}+\frac1{(3i)!}s^3T_0(s^3)+\cdots)+\cdots\\
+s^{3l'_2+2}T_2(s^3)(\frac1{(3i+1)!}+\frac1{(3i)!}s^3T_0(s^3)+\cdots)+\cdots+\cdots\\
+\frac1{(3i+2)!}+\frac1{(3i+1)!}s^3T_0(s^3)+\frac1{(3i)!2}s^6T_0(s^3)^2+\cdots\\
\end{bmatrix}.
\notag
\end{align}
We thus obtain that $(1-\delta\theta_3)t_0^3=1$ and conclude \eqref{LL3}.
Moreover, we have 
\begin{align*}
\min\{3+3l'_1+1,3l_1+1,3l_2+3l'_2+4\}\ge&3l_1+1,\\
\min\{3+3l'_2+2,3l_1+3l'_1+2,3l_2+2\}\ge&3l_2+2.
\end{align*}
By these last two conditions, we can assume  that $l'_1=l_1-1$ and $l'_2=l_2-1$.
Therefore we see that the terms $t^3$, $t^{3l_1+1}$ and $t^{3l_2+2}$ in \eqref{STT} may contribute to 
the coefficients of $s^{3l_1+1}$ and $s^{3l_2+2}$. 
Since $t=t_0s[1+s^3T_0(s^3)+s^{3l_1-2}T_1(s^3)+s^{3l_2-1}T_2(s^3)]$, we have 
\begin{align}
\theta
=&\sum_{i\ge3}\theta_i\frac{t^i}{i!}
=\sum_{i\ge3}{\theta_i}
t_0^is^i\sum_{i_0+i_1+i_2\le i}
s^{3i_0+(3l_1-2)i_1+(3l_2-1)i_2}\frac{T_0(s^3)^{i_0}T_1(s^3)^{i_1}T_2(s^3)^{i_2}}{(i-i_0-i_1-i_2)!i_0!i_1!i_2!}. 
\label{Thetaddd}
\end{align}
We remark that the coefficient of $s^i$ in this series yields $\theta^\delta_i$.

Let us consider $E_{6k}$ (resp.~$E_{6k+2}$) singularities. 
By \eqref{TTii}, we look 
$(i,i_0,i_1,i_2)$ satisfying 
$$
i+3i_0+(3l_1'+1)i_1+(3l_2'+2)i_2\le 3l_2+2\ (\text{resp.}\ 3l_1+1).
$$ 
This holds if and only if 
\begin{equation}\label{IE6k}
(i,i_0,i_1,i_2)=
\begin{cases}
(i,i_0,0,0)\ \text{ with }3i_0\le3(k+p)-i-1\ (\text{resp.} \ 3(k+p)-i+1),\\
(i,i_0,1,0)\ \text{ with }3i_0\le3p+4-i\ (\text{resp.} \ 3p+5-i),\\
(3,0,2,0)\ (\text{resp.} \ (3,0,0,2))\ \text{ with }p=k-2,\\
(3,0,0,1)\ (\text{resp.} \ (3,0,1,0))\ \text{ with }p=k-2.
\end{cases}
\end{equation}
We confirm that the terms $t^3$, $t^{3l_1+1}$ and $t^{3l_2+2}$ in \eqref{Thetaddd} may contribute 
to $\theta^\delta_{3l_1+1}$ and $\theta^\delta_{3l_2+2}$ in \eqref{STT}. 
Moreover, by \eqref{STT}, we obtain
\begin{align*}
\frac{T_1(0)}2=&\delta\,\frac{\theta_{3l_1+1}\,t_0^{3l_1+1}}{(3l_1+1)!}&\Bigl
(\text{resp.}\ \frac{T_2(0)}2=&\delta\,\frac{\theta_{3l_2+2}\,t_0^{3l_2+2}}{(3l_2+2)!}\Bigr), 
\end{align*}
and, by \eqref{Thetaddd} and \eqref{IE6k},
\begin{align*}
\frac{\theta^\delta_{3l_1+1}}{(3l_1+1)!}
=&
\frac{1}{1-\delta\theta_3}\frac{\theta_{3l_1+1}t_0^{3l_1+1}}{(3l_1+1)!}&
\Bigl(\text{resp.}\
\frac{\theta^\delta_{3l_2+2}}{(3l_2+2)!}
=&
\frac{1}{1-\delta\theta_3}\frac{\theta_{3l_2+2}t_0^{3l_2+2}}{(3l_2+2)!}
\Bigr),
\end{align*}
since 
\begin{align*}
\frac{\theta^\delta_{3l_1+1}}{(3l_1+1)!}
=\theta_3t_0^3\frac{T_1(0)}2+\frac{\theta_{3l_1+1}t_0^{3l_1+1}}{(3l_1+1)!}
=&\Bigl(\frac{\delta\theta_3}{1-\delta\theta_3}+1\Bigr)\frac{\theta_{3l_1+1}t_0^{3l_1+1}}{(3l_1+1)!}
=\frac{1}{1-\delta\theta_3}\frac{\theta_{3l_1+1}t_0^{3l_1+1}}{(3l_1+1)!}.
\end{align*}
When $0\le p<k-2$, we obtain
\begin{align*}
\frac{T_2(0)}2=&\delta\,\frac{\theta_{3l_2+2}\,t_0^{3l_2+2}}{(3l_2+2)!}&
\Bigl(\text{resp.}\ \frac{T_1(0)}2=&\delta\,\frac{\theta_{3l_1+1}\,t_0^{3l_1+1}}{(3l_1+1)!}\Bigr), 
\end{align*}
and, by \eqref{Thetaddd} and \eqref{IE6k},
\begin{align*}
\frac{\theta^\delta_{3l_2+2}}{(3l_2+2)!}
=&\frac{1}{1-\delta\theta_3}\frac{\theta_{3l_2+2}t_0^{3l_2+2}}{(3l_2+2)!}&
\Bigl(\text{resp.}\ 
\frac{\theta^\delta_{3l_1+1}}{(3l_1+1)!}
=&\frac{1}{1-\delta\theta_3}\frac{\theta_{3l_1+1}t_0^{3l_1+1}}{(3l_1+1)!}
\Bigr).
\end{align*}
For $E_{6k}$ (resp.~$E_{6k+2}$) singularity with $p=k-2$, we have 
\begin{align*}
3l_1+1=&3k-2\ (\text{resp.}\ 6k-5),&
3l_2+2=&6k-7\ (\text{resp.}\ 3k-1),
\end{align*}
that is, $2(3l_1+1)-3=3l_2+2$ (resp.~$2(3l_2+2)-3=3l_1+1$), 
and we conclude that 
\begin{align*}
\frac{T_1(0)}2
=\delta\,\frac{\theta_{3l_1+1}\,t_0^{3l_1+1}}{(3l_1+1)!}\quad 
\Bigl(\text{resp.}\ \frac{T_1(0)+T_2(0)^2}2=\delta\Bigl(
\frac{\theta_{3l_2+2}t_0^{3l_2+2}}{(3l_2+1)!}T_2(0)+\frac{\theta_{3l_1+1}\,t_0^{3l_1+1}}{(3l_1+1)!}\Bigr)\Bigr), \\
\frac{T_2(0)+T_1(0)^2}2=
\delta\,\Bigl(\frac{\theta_{3l_1+1}t_0^{3l_1+1}}{(3l_1)!}T_1(0)
+\frac{\theta_{3l_2+2}t_0^{3l_2+2}}{(3l_2+2)!}\Bigr)\quad 
\Bigl(\text{resp.}\ \frac{T_2(0)}2=\delta\,\frac{\theta_{3l_2+2}\,t_0^{3l_2+2}}{(3l_2+2)!}\Bigr),
\end{align*}
by \eqref{STT}. We thus obtain that 
\begin{align*}
\theta^\delta_{3l_2+2}
=&
\tfrac{t_0^{3l_2+2}}{1-\delta\theta_3}\Bigl(
\tfrac{\theta_{3l_1+1}^2}{\delta^{-1}-\theta_3}\,
\tfrac{2(3l_2+2)!}{(3l_1)!(3l_1+1)!}
+\theta_{3l_2+2}
\Bigr)
\\
\Bigl(\text{resp.}\ \theta^\delta_{3l_1+1}
=&
\tfrac{t_0^{3l_1+1}}{1-\delta\theta_3}\Bigl(
\theta_{3l_1+1}
+\tfrac{2\theta_{3l_2+2}^2}{\delta^{-1}-\theta_3}\,
\tfrac{(3l_1+1)!}{(3l_2+1)!(3l_2+2)!}
\Bigr)\Bigr), 
\end{align*}
since
\begin{align*}
\frac{\theta^\delta_{3l_2+2}}{(3l_2+2)!}
=&
\theta_3t_0^3\,\frac{T_1(0)^2+T_2(0)}2
+\theta_{3l_1+1}t_0^{3l_1+1}\,\frac{T_1(0)}{(3l_1)!}
+\theta_{3l_2+2}t_0^{3l_2+2}\frac1{(3l_2+2)!}
\\
=&
\frac{\delta\theta_3}{1-\delta\theta_3}\Bigl(\frac{\theta_{3l_1+1}t_0^{3l_1+1}T_1(0)}{(3l_1)!}
+\frac{\theta_{3l_2+2}t_0^{3l_2+2}}{(3l_2+2)!}\Bigr)
+\frac{\theta_{3l_1+1}t_0^{3l_1+1}T_1(0)}{(3l_1)!}
+\frac{\theta_{3l_2+2}t_0^{3l_2+2}}{(3l_2+2)!}
\\
=&
\frac{1}{1-\delta\theta_3}\Bigl(\frac{\theta_{3l_1+1}t_0^{3l_1+1}T_1(0)}{(3l_1)!}
+\frac{\theta_{3l_2+2}t_0^{3l_2+2}}{(3l_2+2)!}\Bigr)
\\
=&
\frac{1}{1-\delta\theta_3}\Bigl(
\frac{2\delta\theta_{3l_1+1}^2t_0^{6l_1+2}}{(3l_1)!(3l_1+1)!}
+\frac{\theta_{3l_2+2}t_0^{3l_2+2}}{(3l_2+2)!}
\Bigr)
\\
=&
\frac{t_0^{3l_2+2}}{1-\delta\theta_3}\Bigl(
\frac1{\delta^{-1}-\theta_3}\,
\frac{2\theta_{3l_1+1}^2}{(3l_1)!(3l_1+1)!}
+\frac{\theta_{3l_2+2}}{(3l_2+2)!}
\Bigr). 
\end{align*}
We conclude the proof.
\end{proof}
\begin{rem}
As in Remark \ref{SexCurv}, we can define the notion of pseudo-curvature 
for an $E_{6k}$ (resp.~$E_{6k+2}$) singularity with $p=k-2$ and $k\ge3$, since $\phi_3^\delta$ 
exhibits an equi-multiple degeneration when $\delta^{-1}$ is equal to
\begin{equation}\label{SecondCurv2}
\theta_3-\frac{2(6k-7)!}{(3k-2)!(3k-3)!}\,\frac{\theta_{3k-2}^2}{\theta_{6k-7}}\quad
\Bigl(\text{resp.~}\theta_3-\frac{2(6k-5)!}{(3k-1)!(3k-2)!}\,\frac{\theta_{3k-1}^2}{\theta_{6k-5}}\Bigr).
\end{equation}
\end{rem}
In summary, the cases in which a degenerate parallel curve has an $\mathcal A$-simple singularity are as follows: 
\begin{center}
\begin{picture}(300,80)
\put(-20,70){\small nonsingular curve}
\put(60,70){\vector(4,-1){30}}\put(80,70){\scriptsize $\theta_2\ne0$}
\put(60,68){\vector(1,-1){30}}
\put(79,50){\scriptsize $\theta_2=0$}
\put(86,44){\scriptsize $\theta_3\ne0$}
\put(60,65){\vector(1,-2){30}}
\put(36,20){\scriptsize $\theta_2=\theta_3=0$}
\put(60,10){\scriptsize $\theta_4\ne0$}
\put(95,60){$A_2$}
\put(95,30){$E_6$}
\put(95,0){$W_{12}$}
\put(150,60){$A_4$}
\put(156,55){\vector(0,-1){14}}
\put(150,30){$E_8$}
\put(180,30){$E_{12}$}
\put(186,25){\vector(0,-1){14}}
\put(180,0){$W_{18}$}
\put(210,60){$A_6$}
\put(216,55){\vector(0,-1){44}}\put(218,35){\scriptsize $\theta_4\ne0$}
\put(210,0){$W^\#_{1,1}$}
\put(240,60){$A_8$}
\put(246,55){\vector(0,-1){44}}\put(248,35){\scriptsize $\theta_4\ne0$}
\put(240,0){$W^\#_{1,3}$}
\put(270,60){$A_{10}$}
\put(276,55){\vector(0,-1){44}}\put(278,35){\scriptsize $\theta_4\ne0$}
\put(270,0){$W^\#_{1,5}$}
\put(300,60){$\cdots$}
\put(300,0){$\cdots$}
\end{picture}\\
\vspace{10pt}
Degenerations of parallel curves at distant $\delta=\theta_m^{-1}$
\end{center}
\begin{rem}
Consider the case $m=4$.

As noted in Remark \ref{rmk:multofP}, the multiplicity of the parallel curve of $W_{12}$ (resp.~$W^{\#}_{1,2q-1}$, $W_{18}$) singularity $\phi_4$ is always $1$ (resp.~$2$, $3$). 
Thus the family of parallel curves $\phi_4^\delta$ has already been treated in Theorem \ref{Thm:Parallel1} (iii) 
(resp.~Theorem \ref{Thm:Parallel2} (iii-b), Theorem \ref{Thm:Parallel3} (iii-b)). 
This implies that $\phi_4^\delta$, for $\delta\ne0$, is nonsingular (resp.~has an $A_{2(q+2)}$ singularity, an $E_{12}$ singularity) at $0$.  

When $\theta_1=\theta_2=\theta_3=0$,  
we have $m^\delta\ge4$ and $t=t_0s+\frac{\delta\theta_5}{10(1-\delta\theta_4)}(t_0s)^2/2+o(t^2)$, 
$t_0=|1-\delta\theta_4|^{-1/4}$, 
$$
\theta^\delta_1=0,\ 
\theta^\delta_2=0,\ 
\theta^\delta_3=0.
$$
In this case, the singularity is not $\mathcal A$-simple. 
\end{rem}

\end{document}